\RequirePackage[l2tabu, orthodox]{nag}
\documentclass[pagesize,pdftex]{article}

\usepackage[utf8]{inputenc}
\usepackage[english]{babel}
\usepackage[dvipsnames]{xcolor}
\usepackage{amsmath,amsfonts,amssymb,amsthm,mathrsfs,cancel,tikz,enumitem,placeins,todonotes,subcaption,mdframed,bbm,mathtools}
\usepackage{graphicx}
\usepackage{pgfplots}
  \pgfplotsset{width=4cm,compat=1.9}
  \usepgfplotslibrary{external}
  \tikzset{->-/.style={decoration={
    markings,
    mark=at position #1 with {\arrow{>}}},postaction={decorate}}} 

\makeatletter
\renewcommand{\todo}[2][]{\tikzexternaldisable\@todo[#1]{#2}\tikzexternalenable}
\makeatother

\usepackage{stmaryrd} 
\usepackage{dsfont}   
\usetikzlibrary{decorations.markings,positioning,snakes}
\usepackage[hidelinks]{hyperref}
\usepackage[left=3cm, right=3cm, top=2.0cm, bottom=2.0cm]{geometry}

\newtheorem{theorem}{Theorem}[section]
\newtheorem{proposition}[theorem]{Proposition}
\newtheorem{lemma}[theorem]{Lemma}
\newtheorem{corollary}[theorem]{Corollary}
\theoremstyle{remark}
\newtheorem{remark}[theorem]{Remark}

\theoremstyle{definition}
\newtheorem{definition}[theorem]{Definition}
\newtheorem{assumption}[theorem]{Assumption}

\DeclareMathAlphabet{\mathpzc}{OT1}{pzc}{m}{it}

\makeatletter
\let\ams@underbrace=\underbrace
\def\underbracel#1_#2{%
  \setbox0=\hbox{$\displaystyle#1$}%
  \ams@underbrace{#1}_{\parbox[t]{\the\wd0}{#2}}%
}
\makeatother

\hypersetup{
  pdfauthor={Jonas Sauer},
  pdftitle={Limited-Range Multilinear Off-Diagonal Extrapolation and Weighted Transference Principle},
  breaklinks=true,
  colorlinks=true,
  linkcolor=blue,
  citecolor=blue,
  urlcolor=blue,
  filecolor=blue,
}

\newcommand{\myint}[3]{\int_{{#1}}{#2}\,\text{d}{#3}}

\newcommand{\eps}{\varepsilon}
 \newcommand{\gambf}{\boldsymbol{\gamma}}
 \newcommand{\om}{{w}}
 \newcommand{\Om}{{v}}
 \newcommand{\ombf}{\boldsymbol{\om}}
 \newcommand{\Ombf}{\boldsymbol{\Om}}

 \newcommand{\qbf}{\mathbf{q}}
 \newcommand{\pbf}{\mathbf{p}}
 \newcommand{\sbf}{\mathbf{s}}
  \newcommand{\tbf}{\mathbf{t}}
 \newcommand{\rbf}{\mathbf{r}}
 
\newcommand{\dd}{\,\mathrm{d}}

\newcommand{\calf}{{\mathcal F}}

\newcommand{\call}{{\mathcal L}}

\newcommand{\calm}{{\mathcal M}}

\newcommand{\calr}{{\mathcal R}}

\newcommand{\calt}{{\mathcal T}}
\newcommand{\calu}{{\mathcal U}}

\newcommand{\R}{\mathbb{R}}

\newcommand{\Z}{\mathbb{Z}}
 \newcommand{\C}{\mathbb{C}}
\newcommand{\N}{\mathbb{N}}

\newcommand{\1}{\mathds{1}}
\DeclareMathOperator{\id}{\mathsf{id}}

\newcommand{\set}[1]{\ensuremath{\{#1\}}}

\newcommand{\setc}[2]{\ensuremath{\{#1\ \lvert\ #2\}}}

\newcommand{\seqkN}[1]{\ensuremath{(#1_k)_{k\in\N}}}

\newcommand{\grp}{\gr}
\newcommand{\dualgrp}{\widehat{\grp}}

\newcommand{\grpH}{\mathfrak{h}}
\newcommand{\dualgrpH}{\widehat{\grpH}}

\newcommand{\gr}{\mathfrak{g}}
\newcommand{\Loq}[2]{{\LR{#2}_{#1}}}
\newcommand{\opnorm}[1]{{\lvert\kern-0.25ex\lvert\kern-0.25ex\lvert #1 \rvert\kern-0.25ex\rvert\kern-0.25ex\rvert}}

\newcommand{\CR}[1]{C^{#1}}  
\newcommand{\LR}[1]{L^{#1}}

\newcommand{\CRc}[1]{\CR{#1}_c}

\newcommand{\tif}{\text{if }}

\newcommand{\cA}{\mathsf{c}_A}
\newcommand{\newCCtr}[2][d]{
\newcounter{#2}\setcounter{#2}{0}
\expandafter\xdef\csname kyedtheconst#2\endcsname{#1}
}
\newcommand{\Cc}[2][nolabel]{
\stepcounter{#2}
\expandafter\ensuremath{\csname kyedtheconst#2\endcsname_{\arabic{#2}}}
\ifthenelse{\equal{#1}{nolabel}}
{}
{\expandafter\xdef\csname kyedconst#1\endcsname
{\expandafter\ensuremath{\csname kyedtheconst#2\endcsname_{\arabic{#2}}}}}
}
\newcommand{\Ccn}[2][nolabel]{
\expandafter\ensuremath{\csname kyedtheconst#2\endcsname}
\ifthenelse{\equal{#1}{nolabel}}
{}
{\expandafter\xdef\csname kyedconst#1\endcsname
{\expandafter\ensuremath{\csname kyedtheconst#2\endcsname}}}
}

\newcommand{\Cclast}[1]{
\expandafter\ensuremath{\csname kyedtheconst#1\endcsname_{\arabic{#1}}}
}

\newcommand{\Ccllast}[1]{
\addtocounter{#1}{-1}
\expandafter\ensuremath{\csname kyedtheconst#1\endcsname_{\arabic{#1}}}
\addtocounter{#1}{1}
}
\newcommand{\const}[1]{
\expandafter{\ifcsname kyedconst#1\endcsname
  \csname kyedconst#1\endcsname
\else
  \errmessage{Undefined Kyedconstant #1.}%
\fi}
}

\title{Limited-Range Multilinear Off-Diagonal Extrapolation and Weighted Transference Principle}
\author{Jonas Sauer\footnote{Institut f\"ur Mathematik und Informatik, Friedrich-Schiller-Universit\"at Jena, Inselplatz 5, 07737 Jena, Germany, \texttt{jonas.sauer@uni-jena.de}}}

\begin{document}



\maketitle

\begin{abstract}
\noindent
Multilinear $L^p$ extrapolation results are established in a limited-range, multilinear, and off-diagonal setting for mixed-norm Lebesgue spaces over $\sigma$-finite measure spaces.
Integrability exponents are allowed in the full range $(0,\infty]$.
We detach the exponents for the weight classes completely from the exponents for the initial and target spaces for the extrapolation except for the basic consistency condition.
This enables us to cover the full range $(0,\infty]$ for all integrability exponents and provides new insights into the dependency of the extrapolated bounds on the weight characteristic.
 Certain endpoint results are new even for $\R^d$.
Additionally, in the setting of compact abelian groups, a weighted transference principle is established.

\smallskip


\end{abstract}

\section{Introduction}\label{tp}
  This article is concerned with extrapolation results in Lebesgue spaces over $\sigma$-finite measure spaces, and in the specific case of locally compact abelian groups with a weighted version of de Leeuw's transference principle \cite{dLe65}, which allows to transfer Fourier multiplier estimates from one group to another. In the unweighted case, a rather general transference principle has been proved in the group setting by Edwards and Gaudry \cite{EdG77}, see also \cite{EiK17} for a modern exposition.
  We prove a generalization to the setting of weighted $L^p$ spaces, where the weights are assumed to be in the Muckenhoupt class $A_p$.
  In fact, we deviate from the usual definition of $A_p$, which helps in formulating the results in a more symmetric way.
  
  \medskip
  
  As a matter of fact, this symmetry also shows its benefits for the extrapolation results.
  The theory of $\LR{p}$-extrapolation has a long history starting with the works by Garc\'{\i}a-Cuerva and Rubio de Francia \cite{GaR85,Rub82}, who realized that boundedness of an operator $T$ on all Muckenhoupt-weighted $\LR{p}$-spaces implies boundedness on all Muckenhoupt-weighted $\LR{q}$-spaces for the range $q\in (1,\infty)$.
  They also showed that vector-valued estimates can be extracted from scalar-valued weighted estimates.
  Later, Buckley \cite{Buc93} found quantitative bounds for the maximal operator in Muckenhoupt weighted spaces in terms of the weight characteristic.
  Restricting the weight classes for the initial spaces of extrapolation, Auscher and Martell proved first versions of limited-range extrapolation in \cite{AuM07}.
  For a detailed history of the theory, we refer to Nieraeth \cite{Nie21}, where a comprehensive theory for multilinear extrapolation is established.
  In particular, in the multilinear case $T:\prod_{j=1}^m\LR{p_j}_{\om_j}\to\LR{p}_{\om}$ with $\frac1p=\sum_{j=1}^m \frac1{p_j}$ some of the $p_j$'s (but not all of them) are allowed to be infinity in the target spaces of the extrapolation.
  In \cite{LMM21}, the authors provided conditions that ensure extrapolation towards $p=\infty$.
  In the linear case, they also showed that this is possible in the off-diagonal setting, that is when $p\ne p_1$.
  Explicit bounds for the results of \cite{LMM21} were recently found in \cite{CLSY24}.
  Recently, Nieraeth \cite{Nie23} found extrapolation results from Lebesgue spaces over $\sigma$-finite measure to general quasi-Banach function spaces, which includes mixed-norm estimates.
  Our contribution to the theory of extrapolation in weighted spaces is as follows:
  \begin{enumerate}
  \item We extend the results of \cite{CLSY24} and \cite{LMM21}: We include the range $(0,1]$ for all initial and target spaces of the extrapolation.
  We show off-diagonal estimates even in the multilinear case and allow for mixed $\LR{\pbf}$-norms and two-weight estimates.
  We provide explicit bounds as in \cite{CLSY24}, which are sharp in all linear and scalar-valued cases considered in \cite{LMM21}.
  In comparison to \cite{CLSY24} and \cite{LMM21} our proofs are significantly shorter and leaner.
  This is achieved by systematically exploiting the symmetry of the weight classes.
  In fact, Theorem \ref{mw_extra_pol_thm_main} contains Theorem 2.3 in \cite{LMM21}, where the authors are forced to consider three different cases, by specialising to the case $p_0\in [1,\infty]$, $p\in (1,\infty]$ and $r_0:=p_0'$, $r=p'$.
  We warn the reader that our $s$ is called $r$ in \cite{LMM21}.
  Finally, we work in the setting of general $\sigma$-finite measure spaces as in \cite{Nie23} instead of $\R^d$.
  \item We show that the approach of Nieraeth developed in \cite{Nie19a,Nie21,Nie23} can be adapted to include the endpoint results from \cite{LMM21}.
  Specialising to weighted mixed-norm spaces, we can provide more explicit bounds than the ones found in \cite{Nie23} in the setting of abstract quasi-Banach function spaces.
  We remark that the possibility of multilinear off-diagonal estimates was mentioned in \cite{Nie23} but not performed explicitly.
  \item
  In contrast to both \cite{LMM21} and the works of Nieraeth, we detach the exponents for the weight classes completely from the exponents for the initial and target spaces for the extrapolation except for the basic consistency condition \eqref{js8}.
  This might be seen as a mere cosmetic change, but it avoids the unpleasant use of negative exponents in \cite{Nie19a,Nie21,Nie23} and crucially helps to understand that it is the exponents of the weight classes (and not the exponents of the Lebesgue spaces) which govern the size of the extrapolation bounds.
  In particular, it is immediately clear from \eqref{js5} that keeping $s_0$, $r_0$, $s$, $r$ fixed and letting $q_0$, $p_0$, $q$, $p$ be constrained by \eqref{js8} but otherwise variable does not change the explicit bounds.
  This insight might serve as an additional heuristic why sharp bounds can not be obtained in general quasi-Banach function spaces.
  On the other hand, keeping $p$ and $r$ separate we can advance to the region $p\in (0,1]$ not covered by \cite{LMM21}, see Remark \ref{js020}\ref{js020iii}.
  \end{enumerate}
 
 In our exposition we follow in many regards the works of Nieraeth.
 Instead of working with extrapolation pairs
 \begin{align*}
 \mathcal{F}\subseteq\{(f,g):f,g:\Omega\to\R\text{ are nonnegative, measurable functions}\}
 \end{align*}
 or multilinear versions of it, we work with the more general map formulation proposed in \cite{Nie23}.
The reason is that this formulation gives a much leaner argument that $Tf$ is well-defined in the target space of the extrapolation, especially for $\LR{\infty}$-type spaces.
Moreover, we work with general $\sigma$-finite measure spaces with a basis of sets.
 Here, a countable collection of measurable sets $\calu$ with $\mu(U)\in (0,\infty)$ for all $U\in\calu$ is a basis of sets in $\Omega$ if
 \begin{enumerate}
 \item $\bigcup_{U\in\calu} U=\Omega$,
 \item $x,y\in\Omega$ $\Rightarrow$ $\exists U\in\calu$ with $x,y\in U$.
 \end{enumerate}
 We fix $(\Omega,\mu)$ together with a basis of sets $U$ for the rest of this article and write $L^0(\Omega;X)$ for the space of strongly measurable functions, where $X$ is a quasi-Banach space.
 In the scalar-valued case, we simply write $L^0(\Omega)$.
 \begin{align*}
  \text{A real-valued function $\om\in L^0(\Omega)$ with $\om>0$ is called a \emph{weight}.}
 \end{align*}
 Given $p\in(0,\infty]$, a weight $\om$, and a quasi-Banach space $X$, we denote by $\LR{p}_{\om}(\Omega;X)$ the space of all $f\in\LR{0}(\Omega;X)$ such that $\|x\mapsto \|f(x)\|_X\om(x)\|_{\LR{p}(\Omega)}<\infty$, where the $\LR{p}(\Omega)$-norm is defined in the usual way.
We again write $\LR{p}_{\om}(\Omega)$ in the scalar-valued case.
 Given $s,r\in(0,\infty]$ and two weights $\om$ and $\Om$, we introduce the weight characteristic
 \begin{equation}\label{mw_pair_def}
  [\om,\Om]_{(s,r)}:=\sup_{U\in\calu}\frac{1}{\mu(U)^{\frac1s+\frac1r}} \|\om\|_{\LR{s}(U)} \|\Om^{-1}\|_{\LR{r}(U)}.
 \end{equation}
 We write $[\om]_{(s,r)}:=[\om,\om]_{(s,r)}$, and for $p\in [1,\infty]$, we write $[\om,\Om]_p:=[\om,\Om]_{(p,p')}$ and $[\om]_p:=[\om,\om]_p$.
 We will discuss properties of weights with finite weight characteristic in Section \ref{expol_multlin} below.
 
 \medskip
 
 For $p\in [1,\infty)$, the quantity
\begin{align*}
[\om]_{A_p}:=\sup_{U}\frac{1}{\mu(U)^p} \|\om\|_{\LR{1}(U)} \|\om^{1-p'}\|_{\LR{1}(U)}^{p-1}=[\om]_{1,p'-1}=[\om^{\frac1p}]_p^p
\end{align*}
is traditionally used in the literature to define the class of Muckenhoupt weights $A_p$ (mostly for $\Omega=\R^d$), together with $A_\infty:=\bigcup_{p\in[1,\infty)} A_p$.
Be aware that $A_\infty$ is not characterized by $[\om]_\infty<\infty$, but rather by finiteness of the Fujii-Wilson constant $[\om]_{A_\infty}:=\sup_{U} \frac{1}{\mu(U)}\int_U \calm (\om \1_U) \dd\mu$.
The more symmetric notion of $[\om]_p$ is advantageous as it often enables to include the endpoint case $p=\infty$ in a straightforward manner.
It also includes the reverse H\"older class $RH_{s}$ introduced in \cite{AuM07}, which is defined by $[\om]_{RH_s}:=[\om,\om^{-1}]_{s,1}<\infty$.
Moreover, it simplifies many statements and calculations as is exemplified by Proposition \ref{mw_weight_prop}\ref{mw_weight_prop_iii} below, whose clean statement should be compared with the corresponding statement (in fact statements!) if the $[\om]_{A_p}$-formalism is employed, see e.g.~Lemma 2.1 in \cite{Duo11}.
 
\medskip

 For $f\in \LR{0}(\Omega)$ we define the \emph{maximal operator} on $\Omega$ via
\begin{align}\label{mw_max_op_def}
 \calm f:=\sup_{U\in \calu} \frac{1}{\mu(U)}\|f\|_{\LR{1}(U)}\1_U.
\end{align}
Note that $\calm f$ is the countable supremum of measurable functions and hence measurable.
Moreover, for $f\in \LR{1}(\Omega)$ we have $\calm f>0$ whenever $f\ne 0$, see \cite[Proposition 2.16]{Nie23}.

\medskip

In order to give a feel for the type of results, we showcase the case of linear one-weight extrapolation, for definiteness with $p_0<p$.
The following result follows directly from the more general Theorem \ref{mw_extra_pol_thm_main} and Remark \ref{js21}\ref{js21iii} below.
\begin{theorem}\label{js200}
Let $q_0,p_0,s_{0},s,r\in (0,\infty)$, $r_{0},q,p\in (0,\infty]$ be such that
\begin{align}\label{js208a}
 \frac1q-\frac1{q_0}=\frac1p-\frac1{p_0}=\frac1s-\frac1{s_0}=\frac1{r_0}-\frac1r<0.
\end{align}
Let $\om$ be a weight with $[\om]_{(s,r)}<\infty$.
 Assume the following:
 \begin{enumerate}
 \item There is $c:(1,\infty]\to[1,\infty)$ such that
\begin{align*}
 \|\calm\|_{\LR{t}_{\Om}(\Omega)\to\LR{t}_{\Om}(\Omega)}\le c(t)[\Om]_{t}^{t'} \quad \text{for all $t\in (1,\infty]$ and weights $\Om$}.
\end{align*}
 \item There is a set $V$, a map $S:V\to\LR{0}(\Omega)$, a map
 \begin{align*}
  T:\bigcup_{[\om_0]_{(s_0,r_0)}<\infty} S^{-1}(\LR{p_{0}}_{\om_{0}}(\Omega))\to \LR{0}(\Omega),
 \end{align*}
 and an increasing function $\phi:\R\to\R$ such that for all weights $\om_{0}$ with $[\om_0]_{(s_0,r_0)}<\infty$ and all $f\in S^{-1}(\LR{p_{0}}_{\om_{0}}(\Omega))$ it holds
\begin{align*}
 \|Tf\|_{\LR{q_0}_{\om_0}}\leq \phi([\om_{0}]_{(s_0,r_0)})\|Sf\|_{\LR{p_{0}}_{\om_{0}}}.
\end{align*}
 \end{enumerate}
Then $Tf$ is well-defined for all $f\in S^{-1}(\LR{p}_{\om}(\Omega))$, and for all $\kappa\in (1,\infty)$ it holds
\begin{align*}
 \|Tf\|_{\LR{q}_{\om}}\le \kappa^{\beta} \phi((\kappa' c(t))^{\beta}[\om]_{(s,r)}^{s/s_0}) \|Sf\|_{\LR{p}_{\om}},
\end{align*}
where $t:=(\frac1s+\frac1r)/\frac1r$ and $\beta:=(\frac1s+\frac1r)(\frac1r-\frac1{r_0})/\frac1r$.
\end{theorem}

\begin{remark}\label{js020}
\begin{enumerate}
\item The same statement holds for $r_0,q,p,s,r\in (0,\infty)$ and $s_0,q_0,p_0\in (0,\infty]$ in the case of downward extrapolation (that is for $>$ instead of $<$ in \eqref{js208a}) if one chooses $t:=(\frac1s+\frac1r)/\frac1s$ and $\beta:=(\frac1s+\frac1r)(\frac1s-\frac1{s_0})/\frac1s$ in the final bound.
\item The boundedness condition for the maximal operator is fulfilled for $\Omega=\R^d$ and more generally in locally compact abelian groups, see Proposition \ref{mw_muckenhoupt_main}.
It is known to be optimal in terms of the exponent in these cases, cf. \cite{Buc93}.
If only boundedness of the maximal operator is known, then Theorem \ref{js200} continues to hold, albeit with a different constant, see Theorem \ref{mw_extra_pol_thm_main}.
\item Some exponents in Theorem \ref{js200} can be chosen to be $\infty$, while others cannot.
There are two different reasons at play that prohibit the choice infinity: One is the fact that Theorem \ref{js200} showcases an instance of upward extrapolation, so that $q_0,p_0<\infty$ (and hence $s_0,r<\infty$ by condition \eqref{js208a}) is merely a manifestation of this choice of direction.
The other reason is that we do not assume the maximal operator to be bounded in $L^1_\om$ for $[\om]_{1}<\infty$ (which is the usual phenomenon known for example for $\Omega=\R^d$).
This translates into the final bound not being finite (or even well-defined since $c(1)$ is not specified) for $s=\infty$.
\item\label{js020i} We emphasize that condition \eqref{js208a} in Theorem \ref{js200} indeed constitutes a limited-range condition:
By $q,p\in (0,\infty]$ and $(s,r)\in (0,\infty)$ it is necessary that
\begin{align*}
\frac1q&\in [\frac{1}{q_0}-\frac{1}{p_0},\infty) \cap (\frac{1}{q_0}-\frac{1}{s_0},\infty) \cap [0,\frac{1}{q_0}+\frac{1}{r_0}),\\
\frac1p&\in [\frac{1}{p_0}-\frac{1}{q_0},\infty) \cap (\frac{1}{p_0}-\frac{1}{s_0},\infty) \cap [0,\frac{1}{p_0}+\frac{1}{r_0}),\\
\frac1s&\in [\frac{1}{s_0}-\frac{1}{q_0},\infty) \cap [\frac{1}{s_0}-\frac{1}{p_0},\infty) \cap (0,\frac{1}{s_0}+\frac{1}{r_0}),\\
\frac1r&\in (0,\frac{1}{q_0}+\frac{1}{r_0}] \cap (0,\frac{1}{p_0}+\frac{1}{r_0}] \cap (0,\frac{1}{s_0}+\frac{1}{r_0}).
\end{align*}
\item\label{js020o} As observed in \cite[Remark 4.8]{Nie23}, the formulation in Theorem \ref{js020} implies a result in terms of extrapolation pairs.
Indeed, if
\begin{align*}
 \mathcal{F}\subseteq\{(f,g):f,g:\Omega\to\R\text{ are nonnegative, measurable functions}\},
\end{align*}
then the choice $V:=\mathcal{F}$, $S(f,g):=f$, and $T(f,g):=g$ shows that a bound $\|g\|_{\LR{q_0}_{\om_0}(\Omega)}\le \phi([\om]_{(s_0,r_0)})\|f\|_{\LR{p_0}_{\om_0}(\Omega)}$ for all $(f,g)\in \mathcal{F}$ and weights $\om_0$ with $[\om_0]_{(s_0,r_0)}<\infty$ extrapolates to a bound $\|g\|_{\LR{q}_{\om}(\Omega)}\le \phi_{s_0,r_0,s}([\om]_{(s,r)})\|f\|_{\LR{p}_{\om}(\Omega)}$ for all $(f,g)\in \mathcal{F}$ and weights $\om$ with $[\om]_{(s,r)}<\infty$.
\item\label{js020ii} The conditions for extrapolating towards $\LR{\infty}_\om$ are necessarily of a rather restrictive type.
This is highlighted by the fact that the Hilbert transform is bounded in $\LR{2}_\om(\R)$ for all weights $\om$ with $[\om]_2<\infty$, but not in $\LR{\infty}(\R)$.
Still, Theorem \ref{js100} shows that an $\LR{2}_{\om_0}(\R)$-bound for all $\om_0$ with $[\om_0]_{(s_0,r_0)}<\infty$ yields an $\LR{\infty}_{\om}(\R)$-bound for all $\om$ with $[\om]_{(s,r)}<\infty$ as long as $s_0,s,r\in (0,\infty)$ and $r_0\in (0,\infty]$
fulfill $\frac{1}{s}-\frac{1}{s_0}=\frac{1}{r_0}-\frac1r=-\frac12$.
A possible choice is $s_0:=r:=\frac65$, $s:=r_0:=3$.
\item\label{js020iii} We can deduce a bound $\LR{1}_\om(\R^d)\to\LR{\infty}_\om(\R^d)$ for all weights $\om$ with $[\om]_2<\infty$ whenever we already know a bound $\LR{\frac23}_{\om_0}(\R^d)\to \LR{2}_{\om_0}(\R^d)$ for all weights $\om_0$ with $[\om_0]_{1}<\infty$.
Indeed, setting
\begin{align*}
q_0&:=2, \ p_0:=\frac23, \ s_0:=1, \ r_0:=\infty, \ q:=\infty, \ p:=1, \ s:=2, \ r:=2,
\end{align*}
we obtain this assertion as a particular instance of Theorem \ref{js200} since $\|\calm\|_{\LR{t}_{\om}(\R^d)\to\LR{t}_{\om}(\R^d)}\le c(t)[\om]_{t}^{t'}$ for all $t\in (1,\infty]$ and weights $\om$.
This shows that by detaching the exponents for the weight classes from the exponents for the initial and target spaces for the extrapolation, Theorem \ref{js200} can cover situations that can not be treated by \cite[Theorem 2.3]{LMM21}, where $p$ has to be strictly larger than $1$.
\end{enumerate}
\end{remark}
The paper is organized as follows.
In Section 2, we provide basic properties of the weight characteristic and prove a limited-range multilinear off-diagonal extrapolation result.
In Section 3, we generalize these results to mixed-norm spaces and give certain vector-valued extensions. In particular, we obtain explicit $\calr$-bounds of families of linear operators from weighted bounds.
Finally, in Section 4 we prove a weighted version of de Leeuw's transference principle in the context of locally compact abelian groups.

\section{Limited-Range Multilinear Off-Diagonal Extrapolation}\label{expol_multlin}
We start with basic properties of the weight characteristic.
\begin{proposition}\label{mw_weight_prop}
Let $w$ and $v$ be weights and $s,r\in (0,\infty]$. 
\begin{enumerate}
\item\label{mw_weight_prop_i} $[\om,\Om]_{(s,r)}=[\Om^{-1},\om^{-1}]_{(r,s)}$.
\item\label{mw_weight_prop_ii} $[\om^\alpha,\Om^\alpha]_{(\frac{s}{\alpha},\frac{r}{\alpha})}=[\om,\Om]_{(s,r)}^\alpha$ for all $\alpha\in(0,\infty)$.
\item\label{mw_weight_prop_iii} Let $\om_1,\ldots,\om_m$ and $\Om_1,\ldots,\Om_m$ be weights and $s_1,\ldots,s_m,r_1,\ldots,r_m\in (0,\infty]$ such that $\om=\prod_{j=1}^m \om_j$, $\Om=\prod_{j=1}^m \Om_j$, $\frac1s=\sum_{j=1}^m \frac1{s_j}$, and $\frac1r=\sum_{j=1}^m \frac1{r_j}$.
Then it holds
  \begin{align*}
   [\om,\Om]_{(s,r)}\le [\om_1,\Om_1]_{(s_1,r_1)}\cdots [\om_m,\Om_m]_{(s_m,r_m)}.
  \end{align*}
 In particular $[w,v]_{(s,r)}\le [\om,\Om]_{(s_1,r_1)}$ whenever $s\le s_1, r\le r_1$.
\item\label{mw_weight_prop_iv} $[\om,\Om]_{(1,\infty)}=\|\calm w\|_{\LR{\infty}_{\Om^{-1}}(\Omega)}=\|\calm\|_{\LR{\infty}_{\om^{-1}}(\Omega)\to\LR{\infty}_{\Om^{-1}}(\Omega)}$.
\item\label{mw_weight_prop_v} $[\om,\Om]_{(\infty,1)}=\|\calm \Om^{-1}\|_{\LR{\infty}_{\om}(\Omega)}=\|\calm\|_{\LR{\infty}_\Om(\Omega)\to\LR{\infty}_{\om}(\Omega)}$.
\end{enumerate}
\end{proposition}
\begin{proof}
The assertions in \ref{mw_weight_prop_i} and \ref{mw_weight_prop_ii} follow trivially from the definition of $[\om,\Om]_{(s,r)}$ in \eqref{mw_pair_def}.
Let $U\in\calu$.
By H\"older's inequality we have
\begin{align*}
\frac1{\mu(U)^{\frac1s+\frac1r}}\|\om\|_{\LR{s}(U)}\|\Om^{-1}\|_{\LR{r}(U)}
\le \prod_{j=1}^m \frac1{\mu(U)^{\frac1{s_j}+\frac1{r_j}}}\|\om_j\|_{\LR{s_j}(U)}\|\Om_j^{-1}\|_{\LR{r_j}(U)}.
\end{align*}
Taking the supremum over all $U\in\calu$ yields the product estimate in \ref{mw_weight_prop_iii}.
The additional assertion follows by taking $m=2$ and $\om_2=\Om_2=1$ due to $[1,1]_{(s_2,r_2)}=1$.
For \ref{mw_weight_prop_iv} let $x\in \Omega$.
Then for all $U\in \calu$ we have $\1_U(x)\Om(x)^{-1}\le \|\Om^{-1}\|_{\LR{\infty}(U)}$ and thus
\begin{align*}
\frac1{\mu(U)}\|\om\|_{\LR{1}(U)}\1_U(x)\Om(x)^{-1}\le \frac1{\mu(U)}\|\om\|_{\LR{1}(U)}\|v^{-1}\|_{\LR{\infty}(U)}\le [\om,\Om]_{(1,\infty)}.
\end{align*}
Taking the supremum over all $U\in \calu$ and $x\in \Omega$, we obtain $\|\calm \om\|_{\LR{\infty}_{\Om^{-1}}(\Omega)}\le [\om,\Om]_{(1,\infty)}$.
On the other hand, let $U\in \calu$ and $\eps>0$.
Choose $x\in U$ with $\|\Om^{-1}\|_{\LR{\infty}(U)}\le (1+\eps)/\Om(x)$.
Then
\begin{align*}
\frac1{\mu(U)}\|\om\|_{\LR{1}(U)}\|\Om^{-1}\|_{\LR{\infty}(U)}
\le  \frac{(1+\eps)}{\mu(U)}\|\om\|_{\LR{1}(U)}/\Om(x)
\le (1+\eps)\|(\calm\om)/\Om\|_{\LR{\infty}(\Omega)}.
\end{align*}
Since $\eps>0$ and $U\in\calu$ were arbitrary, this gives $[\om,\Om]_{(1,\infty)}\le \|(\calm\om)/\Om\|_{\LR\infty(\Omega)}$.
Together this yields $[\om,\Om]_{(1,\infty)}= \|(\calm\om)/\Om\|_{\LR\infty(\Omega)}$.
For the final equality, observe that $\|\calm\om\|_{\LR{\infty}_{\Om^{-1}}(\Omega)}\le \|\calm\|_{\LR{\infty}_{\om^{-1}}(\Omega)\to\LR{\infty}_{\Om^{-1}}(\Omega)}$ due to $\|\om\|_{\LR{\infty}_{\om^{-1}}(\Omega)}=1$, and that
the sublinearity of $\calm$ yields
\begin{align*}
(\calm f)/\Om = \calm(f\om^{-1}\om)/\Om\le \calm(\om)/\Om\|f\|_{\LR{\infty}_{\om^{-1}}(\Omega)}\le \|\calm\om\|_{\LR{\infty}_{\Om^{-1}}(\Omega)}\|f\|_{\LR{\infty}_{\om^{-1}}(\Omega)},
\end{align*}
which shows $\|\calm\|_{\LR{\infty}_{\om^{-1}}(\Omega)\to\LR{\infty}_{\Om^{-1}}(\Omega)}\le \|\calm\om\|_{\LR{\infty}_{\Om^{-1}}(\Omega)}$.
Part \ref{mw_weight_prop_v} follows from \ref{mw_weight_prop_iv} and \ref{mw_weight_prop_i}.
\end{proof}
We remark that Proposition \ref{mw_weight_prop} includes the interpolation assertion
\begin{align*}
[\om_0^{1-\theta}\om_1^\theta,\Om_0^{1-\theta}\Om_1^\theta]_{(s,r)}\le [\om_0,\Om_0]_{(s_0,r_0)}^{1-\theta}[\om_1,\Om_1]_{(s_1,r_1)}^\theta
\end{align*}
for $s_0,s_1,r_0,r_1\in (0,\infty]$, $\theta\in [0,1]$ and $\frac1s:=\frac{1-\theta}{s_0}+\frac\theta{s_1}$, $\frac1r:=\frac{1-\theta}{r_0}+\frac{\theta}{r_1}$.

\medskip

The following lemma is the key construction in proving extrapolation theorems for weighted spaces.
We extend an approach by Nieraeth \cite{Nie19a}, which crucially uses the symmetry of the definition of $[\om]_{(s_0,r_0)}$ to circumvent the necessity of giving two different arguments for $p<p_0$ and $p>p_0$.
Even for $\Omega=\R^d$, the formulation with $u=\infty$ or $p=\infty$ seems to be new.
Its formulation is facilitated by the following definition of weight classes and rescaled exponents.
\begin{definition}\label{js212}
Let $\frac1{\gamma}\in\R$, $s_0,r_0\in (0,\infty]$, and let $\om$ and $\Om$ be weights.
Then we introduce
\begin{align*}
(t_0,t,\om_{t},\Om_{t},W)
\end{align*}
in the following way: Define $s,r\in (0,\infty]$ via $\frac1{s}-\frac1{s_{0}}=\frac1{r_{0}}-\frac1{r}=\frac1{\gamma}$ and set $\frac1\alpha:=\frac1{s_0}+\frac1{r_0}$.
Moreover, set
\begin{align}\label{js7e1}
\begin{cases}
 t_0:=\frac{s_{0}}{\alpha}, \quad t:=\frac{s}{\alpha}, \quad (\om_{t},\Om_{t}):=(\om^\alpha, \Om^\alpha) & \text{if } \frac1{\gamma}>0,\\
 t_0:=1, \quad t:=1, & \text{if } \frac1{\gamma}=0,\\
 t_0:=\frac{r_{0}}{\alpha}, \quad t:=\frac{r}{\alpha}, \quad (\om_{t},\Om_{t}):=(\Om^{-\alpha},\om^{-\alpha}) & \text{if } \frac1{\gamma}<0.
\end{cases}
\end{align}
and
\begin{align}\label{js204}
W:=
\begin{cases}
 \set{(\om,\Om)} & \text{if } \frac1{\gamma}=0, \\
 \setc{(\tilde\om,\tilde\Om)}{\tilde\om\tilde\Om^{-1}=\om\Om^{-1} \text{ and } [\tilde\om,\tilde\Om]_{(s_0,r_0)}<\infty} & \text{else.} 
\end{cases}
\end{align}
For the case $\frac1{\gamma}=0$ we do not introduce $(\om_{t},\Om_{t})$, but use the following convention: $\calm$ is considered to be bounded from $\LR{t}_{\Om_{t}}(\Omega)$ to $\LR{t}_{\om_{t}}(\Omega)$.
\end{definition}
At first glance, the convention for $\frac{1}{\gamma} = 0$ may appear unusual.
However, it is specifically designed to ensure that the statements remain valid even in trivial cases, thereby eliminating the need for cumbersome case distinctions that could otherwise distract from the core arguments.
Moreover, if $\om=\Om$, the definition of the weight class $W$ ensures that $\om_0=\Om_0$ for $(\om_0,\Om_0)\in W$.
\begin{lemma}\label{js7}
Let
$u_0,p_0,s_0,r_0,u,p,s,r\in (0,\infty]$ be such that
\begin{align*}
\frac1{u_0}-\frac1u=\frac1p-\frac1{p_0}=\frac1s-\frac1{s_0}=\frac1{r_0}-\frac1r=:\frac1\gamma.
\end{align*}
Let $\om$, $\Om$ be weights with $[\om,\Om]_{(s,r)}<\infty$, and define
\begin{align*}
 (t_0, t, \om_t,\Om_t, W) \text{ via Definition \ref{js212}}.
\end{align*}
If $\calm$ is bounded from $\LR{t}_{\Om_t}(\Omega)$ to $\LR{t}_{\om_t}(\Omega)$, then for all $f\in \LR{p}_\Om(\Omega)$, $h\in \LR{u}_{\om^{-1}}(\Omega)$ there are weights $(\om_0,\Om_0)\in W$ such that $f\in \LR{p_0}_{\Om_0}(\Omega), \ h\in \LR{u_0}_{\om_0^{-1}}(\Omega)$.\\
For every $\kappa\in(1,\infty)$, these weights can be chosen such that
\begin{align}\label{js7_e1}
[\om_0,\Om_0]_{(s_0,r_0)}&\le (\kappa'\|\calm\|_{\LR{t}_{\Om_t}\to\LR{t}_{\om_t}})^{\frac t{|\gamma|}}[\om,\Om]_{(s,r)}^{t/t_0},
\end{align}
as well as
\begin{align*}
\|f\|_{\LR{p_0}_{\Om_0}}\|h\|_{\LR{u_0}_{\om_0^{-1}}}\le \kappa^{\frac t{|\gamma|}}\|f\|_{\LR{p}_\Om}\|h\|_{\LR{u}_{\om^{-1}}}.
\end{align*}
\end{lemma}
\begin{proof}
For $\frac1{\gamma}=0$, the assertion is trivial with $(\om_0,\Om_0):=(\om,\Om)$.
Thus, it suffices to assume $\gamma\in (0,\infty)$ by symmetry in $\gamma\leftrightarrow -\gamma$, $u_0\leftrightarrow p_0$, $r_0\leftrightarrow s_0$, $u\leftrightarrow p$, $s\leftrightarrow r$, $\om_0 \leftrightarrow \Om_0^{-1}$, and $\om\leftrightarrow \Om^{-1}$.
Moreover, replacing $(\om,\Om)$ by $(\om^\alpha,\Om^\alpha)$ if necessary, we may assume $\alpha=1$ by Proposition \ref{mw_weight_prop}\ref{mw_weight_prop_ii}.
In particular, $p$ is finite, $s=t\in [1,\infty)$ and $s_0=t_0\in [1,\infty]$ with $s'=r$ and $s_0'=r_0$.
Let $f\in \LR{p}_\Om(\Omega)$, $h\in \LR{u}_{\om^{-1}}(\Omega)$, and $\kappa\in(1,\infty)$.
We assume that $f\ne 0$ (otherwise replace $f$ by some nonzero element in $\LR{p}_\Om(\Omega)$).

\medskip

Define $H$ via $H^{s}\om^{s}=|f|^p\Om^p$ and with the notation $\calm_t g:=\calm g_t$, where $g_t:=g\om\Om^{-1}$,
 \begin{align*}
 R:=\sum_{k=0}^\infty\frac{\calm_t^k H}{(\kappa'\|\calm_t\|_{\LR{s}_{\om}\to\LR{s}_{\om}})^k}.
\end{align*}
Since $\|\calm_t\|_{\LR{s}_{\om}\to\LR{s}_{\om}}=\|\calm\|_{\LR{s}_{\Om}\to\LR{s}_{\om}}$, it follows from this definition and Proposition \ref{mw_weight_prop}\ref{mw_weight_prop_iv} that
\begin{enumerate}
\item\label{js7ii} $R>0$ and $R^{-1}\le H^{-1}$ almost everywhere,
\item\label{js7i} $\calm R_t\le \kappa'\|\calm\|_{\LR{s}_{\Om}\to\LR{s}_{\om}}R$ for $R_t:=R\om\Om^{-1}$, i.e. $[R_t,R]_{(1,\infty)}\le \kappa'\|\calm\|_{\LR{s}_{\Om}\to\LR{s}_{\om}}$,
\item\label{js7iii} $\|R\|_{\LR{s}_{\om}}\le \kappa\|H\|_{\LR{s}_{\om}}$.
\end{enumerate}
Let $\om_0:=R^{-\frac s{\gamma}}\om^{\frac s{s_0}}$ and $\Om_0:=R_t^{-\frac{s}{\gamma}}\Om^{\frac{s}{s_0}}$, so that $\om_0\Om_0^{-1}=\om\Om^{-1}$ in virtue of $\frac{s}{\gamma}+\frac{s}{s_0}=1$.
Since $\frac{s}{\gamma}+\frac{s}{s_0 r}=\frac{1}{r_0}$, Proposition \ref{mw_weight_prop} shows
\begin{align}\label{js7_e2}
[\om_0,\Om_0]_{(s_0,r_0)}
\le [R^{-\frac{s}{\gamma}},R_t^{-\frac{s}{\gamma}}]_{(\infty,\frac{\gamma}{s})} [\om^{\frac{s}{s_0}},\Om^{\frac{s}{s_0}}]_{(s_0,s_0 r/s)} 
= [R_t,R]_{(1,\infty)}^{\frac{s}{\gamma}}[\om,\Om]_{(s,r)}^{\frac{s}{s_0}}.
\end{align}
Thus, \eqref{js7_e1} follows from \eqref{js7_e2} via \ref{js7i}.
On $\setc{x\in\Omega}{|f(x)|>0}$, \ref{js7ii} shows $|f|\Om_0\le |f| (H\om\Om^{-1})^{-\frac{s}{\gamma}} \Om^{\frac{s}{s_0}}=(|f|\Om)^{\frac{p}{p_0}}$, so that
\begin{align*}
\|f\Om_0\|_{p_0} \le  \|f\Om\|_{p}^{\frac{p}{p_0}}.
\end{align*}
Moreover, H\"older's inequality yields
\begin{align*}
\|h\om_0^{-1}\|_{u_0}\le \|h\om^{-1}\|_{u}\|\om_0^{-1}\om\|_{\gamma}
=\|h\om^{-1}\|_{u}\|R\om\|_{s}^{\frac{s}{\gamma}},
\end{align*}
so that the assertion follows since $\|R\om\|_{s}^{\frac{s}{\gamma}}\le \kappa^{\frac{s}{\gamma}}\|H\om\|_{s}^{\frac{s}{\gamma}} = \kappa^{\frac{s}{\gamma}}\|f\Om\|_p^{1-\frac{p}{p_0}}$ by \ref{js7iii}.
\end{proof}
\begin{remark}\label{js21}
\begin{enumerate}
\item\label{js21ii} While Lemma \ref{js7} holds for $s,r\in (0,\infty]$, in applications it is often limited to $s,r\in (0,\infty)$.
Clearly, if both $s=r=\infty$, then $1/\gamma=0$ and thus the assertion is trivial.
If only one exponent is infinite, say $r=\infty$, then $t=1$, and hence $\calm$ needs to be bounded from $\LR{1}_{\Om^s}(\Omega)$ to $\LR{1}_{\om^s}(\Omega)$ with $[\om^s,\Om^s]_1=[\om,\Om]_{(s,\infty)}^s<\infty$, which is typically not the case.
\item\label{js21iii} If there is $c:(1,\infty]\to[1,\infty)$ such that
\begin{align*}
 \|\calm\|_{\LR{p}_{\om}(\Omega)\to\LR{p}_{\om}(\Omega)}\le c(p)[\om]_{p}^{p'} \quad \text{for all $p\in (1,\infty]$ and weights $\om$,}
\end{align*}
then the boundedness condition for $\calm$ in Lemma \ref{js7} is fulfilled automatically in the case $s,r\in (0,\infty)$ and $\om=\Om$, since $t\in(1,\infty]$ and $[w_t]_t^{t'}=[w]_{(s,r)}^{\alpha t'}<\infty$.
Moreover, $\alpha t' \frac{t}{|\gamma|} + \frac{t}{t_0}=\frac{t'}{t_0'}$, so that estimate \eqref{js7_e1} turns into
\begin{align*}
[\om_0]_{(s_0,r_0)}&\le (\kappa'c(t))^{\frac t{|\gamma|}} [\om]_{(s,r)}^{t'/t_0'} \quad \text{ for all }\kappa\in(1,\infty).
\end{align*}
We mention that $\frac{t'}{t_0'}=\frac{r}{r_0}$ if $\gamma>0$ and $\frac{t'}{t_0'}=\frac{s}{s_0}$ if $\gamma<0$.
\end{enumerate}
\end{remark}
As an immediate corollary to Lemma \ref{js7} we obtain that a weighted $\LR{p}$ space is included in the union over a certain class of weighted $\LR{p_0}$ spaces.
\begin{corollary}\label{js201}
Let $p_0,s_0,r_0,p,s,r\in (0,\infty]$ be such that
\begin{align*}
 \frac1p-\frac1{p_0}=\frac1s-\frac1{s_0}=\frac1{r_0}-\frac1r=:\frac1\gamma.
\end{align*}
Let $\om$ and $\Om$ be weights with $[\om,\Om]_{(s,r)}<\infty$, and define
\begin{align*}
 (t_0, t, \om_t,\Om_t, W) \text{ via Definition \ref{js212}}.
\end{align*}
If $\calm$ is bounded from $\LR{t}_{\Om_t}(\Omega)$ to $\LR{t}_{\om_t}(\Omega)$, then it holds
\begin{align*}
\LR{p}_\Om(\Omega)\subseteq \bigcup_{(\om_0,\Om_0)\in W}\LR{p_0}_{\Om_0}(\Omega).
\end{align*}
\end{corollary}

We now obtain a limited-range off-diagonal extrapolation result for multilinear operators.
In order to formulate the result, we introduce the following notation.
We fix $m\in\N$ and write $J:=\set{1,\ldots,m}$.
For $\vec{\gamma}\in (0,\infty]^m$ write $\frac1{\gamma}:=\sum_{j\in J} \frac1{\gamma_j}$.
Similarly, we write $\vec{\om}=(\om_{1},\ldots,\om_{m})$ for weights $\om_{j}$, $j\in J$, and introduce the product weight $\om:=\prod_{j\in J} \om_{j}$.
We set $[\vec{\om},\vec{\Om}]_{(\vec{s},\vec{r})}:=([\om_{j},\Om_{j}]_{(s_{j},r_{j})})_{j\in J}$ and write $[\vec{\om},\vec{\Om}]_{(\vec{s},\vec{r})}<\infty$ if $[\om_{j},\Om_{j}]_{(s_{j},r_{j})}<\infty$ for all $j\in J$.
\begin{theorem}\label{mw_extra_pol_thm_main}
 Let $q_i\in (0,\infty]$ and $\vec{p}_i,\vec{s}_i,\vec{r}_i\in (0,\infty]^m$ be such that
\begin{align}\label{js8}
  \frac{1}{\vec{p}_1}-\frac{1}{\vec{p}_0}=\frac{1}{\vec{s}_1}-\frac{1}{\vec{s}_0}=\frac{1}{\vec{r}_0}-\frac{1}{\vec{r}_1}=:\frac1{\vec{\gamma}}, \quad \frac1{q_1}-\frac1{q_0}=\frac1\gamma.
\end{align}
 Let $\vec{\om}_{1}$, $\vec{\Om}_{1}$ be vectors of weights with $[\vec{\om}_{1},\vec{\Om}_{1}]_{(\vec{s}_{1},\vec{r}_{1})}<\infty$, and define for $j\in J$
 \begin{align*}
 (t_{0j}, t_{1j}, \om_{t_{1j}},\Om_{t_{1j}}, W_j) \text{ via Definition \ref{js212}}
\end{align*}
  with $\gamma:=\gamma_{j}$, $s_0:=s_{0j}$, $r_0:=r_{0j}$, $\om:=\om_{1j}$, $\Om:=\Om_{1j}$.
  Assume the following:
  \begin{enumerate}
 \item $\calm$ is bounded from $\LR{t_{1j}}_{\Om_{t_{1j}}}(\Omega)$ to $\LR{t_{1j}}_{\om_{t_{1j}}}(\Omega)$ for all $j\in J$.
  \item For all $j\in J$ there is a set $V_j$, a map $S_j:V_j\to\LR{0}(\Omega)$, a map
   \begin{align*}
  T:\prod_{j\in J}\bigcup_{(\om_{0j},\Om_{0j})\in W_j} S^{-1}(\LR{p_{0j}}_{\Om_{0j}}(\Omega))\to \LR{0}(\Omega),
 \end{align*}
 and a function $\phi:\R^m\to\R$ increasing in each component such that for all vectors of weights $(\vec{\om}_{0}, \vec{\Om}_{0})\in \vec{W}:=W_1\times\cdots\times W_m$ and for all $\vec{f}\in \prod_{j\in J} S_j^{-1}(\LR{p_{0j}}_{\Om_{0j}}(\Omega))$ it holds
\begin{align}\label{js010a}
 \|T\vec{f}\|_{\LR{q_0}_{\om_0}}\leq \phi([\vec{\om}_{0},\vec{\Om}_{0}]_{(\vec{s}_0,\vec{r}_0)})\prod_{j\in J}\|S_jf_j\|_{\LR{p_{0j}}_{\Om_{0j}}}.
\end{align}
  \end{enumerate}
Then $T\vec{f}$ is well-defined for all $\vec{f}\in \prod_{j\in J} S_j^{-1}(\LR{p_{1j}}_{\Om_{1j}}(\Omega))$, and for all $\kappa\in (1,\infty)$ it holds
\begin{align}\label{js5}
 \|T\vec{f}\|_{\LR{q_{1}}_{\om_1}}\le \kappa^{\beta} \phi(\vec{C}_\kappa) \prod_{j\in J}\|S_jf_j\|_{\LR{p_{1j}}_{\Om_{1j}}},
\end{align}
where $\beta:=\sum_{j\in J} \frac {t_{1j}}{|\gamma_j|}$ and $C_{\kappa j}:=(\kappa'\|\calm\|_{\LR{t_{1j}}_{\Om_{t_{1j}}}\to\LR{t_{1j}}_{\om_{t_{1j}}}})^{\frac {t_{1j}}{|\gamma_j|}}[\om_{1j},\Om_{1j}]_{(s_{1j},r_{1j})}^{t_{1j}/t_{0j}}$.
\end{theorem}
\begin{proof}
Let $\vec{f}\in \prod_{j\in J} S_j^{-1}(\LR{p_{1j}}_{\Om_{1j}}(\Omega))$ and $\kappa\in(1,\infty)$.
Then $T\vec{f}$ is well defined by Corollary \ref{js201}.
Choose $\frac1\lambda\in(0,\infty)$ sufficiently large such that $\frac1\lambda > \frac1{q_0}+\frac{m}{|\gamma_j|}$ for all $j\in J$.
Then $\tilde q_0,\tilde q_1\in (1,\infty]$ for $\tilde{q}_0:=q_0/\lambda$ and $\tilde q_1:=q_1/\lambda$.
For $j\in J$ we define $u_{0j}:=m \lambda \tilde{q}_0'$ and $u_{1j}\in (0,\infty]$ via $\frac{1}{u_{0j}}-\frac{1}{u_{1j}}=\frac{1}{\gamma_j}$, so that $\sum_{j\in J} \frac{\lambda}{u_{1j}}=\frac1{\tilde{q}_1'}$ by \eqref{js8}.
Thus for $h\om_1^{-\lambda}\in \LR{\tilde{q}_1'}(\Omega)$ with unit norm, the choice $h_j:=\om_{1j} (|h|\om_1^{-\lambda})^{\tilde{q}_1'/u_{1j}}$ yields $0\le h_j\om_{1j}^{-1}\in \LR{u_{1j}}(\Omega)$ with unit norm and $|h|= \prod_{j\in J} h_j^\lambda$ due to $\om_1:=\prod_{j\in J} \om_{1j}$.

\medskip

\noindent
For $j\in J$, apply Lemma \ref{js7} with the set of parameters $(u_0,p_0,s_0,r_0,u,p,s,r)$ replaced by the set $(u_{0j}, p_{0j}, s_{0j}, r_{0j},u_{1j}, p_{1j}, s_{1j}, r_{1j})$ to the weights $\om_{1j}$, $\Om_{1j}$ and functions $S_jf_j$ and $h_j$.
This yields weights $(\vec{\om}_{0}, \vec{\Om}_{0})\in \vec{W}$ such that $[\vec{\om}_{0},\vec{\Om}_{0}]_{(\vec{s}_0,\vec{r}_0)}\le \vec{C}_\kappa$ componentwise, and by H\"older's inequality (recall $\om_0:=\prod_{j\in J} \om_{0j}$)
\begin{align*}
\Big|\int_\Omega |T\vec{f}|^\lambda h \dd\mu \Big| &\le \||T\vec{f}|^\lambda\om_0^\lambda\|_{\LR{\tilde{q}_0}}\|h\om_0^{-\lambda}\|_{\LR{\tilde{q}_0'}} = \|T\vec{f}\|_{\LR{q_0}_{\om_0}}^\lambda\|h\om_0^{-\lambda}\|_{\LR{\tilde{q}_0'}} \\
&\le \phi([\vec{\om}_0,\vec{\Om}_0]_{(\vec{s}_0,\vec{r}_0)})^\lambda \prod_{j\in J} \|S_j f_j\|_{\LR{p_{0j}}_{\Om_{0j}}}^\lambda\|h_j\om_{0j}^{-1}\|_{\LR{u_{0j}}}^\lambda \\
&\le \kappa^{\lambda\sum_{j\in J}\frac{t_{1j}}{|\gamma_j|}}\phi(\vec{C})^\lambda \prod_{j\in J} \|S_jf_j \|_{\LR{p_{1j}}_{\Om_{1j}}}^\lambda.
\end{align*}
Since $\|T\vec{f}\|_{\LR{q_1}_{\om_1}}^\lambda=\||T\vec{f}|^\lambda\om_1^\lambda\|_{\LR{\tilde{q}_1}}$ is attained by taking the supremum of $\left|\int_\Omega |Tf|^\lambda h \dd\mu \right|$ over all $h\om_{1}^{-\lambda}\in \LR{\tilde{q}_1'}(\Omega)$ with unit norm, we obtain \eqref{js5}.
\end{proof}
\begin{remark}\label{js202}
Let there be $c:(1,\infty]\to [1,\infty)$ such that
\begin{align*}
 \|\calm\|_{\LR{p}_{\om}(\Omega)\to\LR{p}_{\om}(\Omega)}\le c(p)[\om]_{p}^{p'} \quad \text{for all $p\in (1,\infty]$ and weights $\om$.}
\end{align*}
Then by Remark \ref{js21}\ref{js21iii} the boundedness condition for $\calm$ in Theorem \ref{mw_extra_pol_thm_main} is fulfilled if $s_{1j},r_{1j}\in(0,\infty)$ and $\om_{1j}=\Om_{1j}$ for all $j\in J$ with $\frac1{\gamma_j}\ne 0$, and we have for all $\kappa\in (1,\infty)$
\begin{align}\label{js202_e1}
 C_j\le (\kappa'c(t_{1j}))^{\frac {t_{1j}}{|\gamma_j|}}[\om_{1j},\Om_{1j}]_{(s_{1j},r_{1j})}^{t_{1j}'/t_{0j}'}.
\end{align}
\end{remark}

\section{Limited-Range Multilinear Off-Diagonal Mixed-Norm Extrapolation}

In this section we extend Theorem \ref{mw_extra_pol_thm_main} into several directions: In Theorem \ref{mw_extra_pol_prop} we include multilinear off-diagonal extrapolation in mixed-norms, and can allow for certain vector-valued extensions.
To this end we need to introduce some notation.
We fix $\ell\in \N$ and $m\in\N$.
The letters $I$ and $J$ will always denote the sets $I:=\set{1,\ldots,\ell}$ and $J:=\set{1,\ldots,m}$, respectively.
We also fix a $\sigma$-finite measure spaces $(\Omega_i,\mu_{i})$, $i\in I$.

\medskip

For \emph{exponents} $p_{ij}\in(0,\infty]$, $(i,j)\in I\times J$, we consider \emph{vectors} (columns) \emph{of exponents} $\pbf_j=(p_{1j},\ldots,p_{\ell j})\in(0,\infty]^\ell$ which are related to mixed norms, \emph{tuples} (rows) \emph{of exponents} $\vec{p}_i=(p_{i1},\ldots,p_{im})\in (0,\infty]^m$ which are related to multilinearity, and \emph{vector-tuples} (matrices) \emph{of exponents} $\vec{\pbf}=(p_{ij})_{(i,j)\in I\times J}\in (0,\infty]^{\ell\times m}$ which combine the two concepts.
Similarly, for weights $\om_{ij}$ on $\Omega_i$, $(i,j)\in I\times J$, we consider \emph{vectors of weights} $\ombf_j=(\om_{1j},\ldots,\om_{\ell j})$, \emph{tuples of weights} $\vec{\om}_i=(\om_{i1},\ldots,\om_{im})$ on $\Omega_i$, and \emph{vector-tuples of weights} $\vec{\ombf}:=(\om_{ij})_{(i,j)\in I\times J}$.
For vector-tuple of weights $\vec{\ombf}$, $\vec{\Ombf}$ and $\vec{\sbf},\vec{\rbf}\in (0,\infty]^{\ell\times m}$ we introduce $\frac1{s_i}:=\sum_{j=1}^m \frac1{s_{ij}}$, $\sbf:=(s_1,\ldots,s_\ell)$, the product weight $\ombf:=(\prod_{j\in J}\om_{1j},\ldots,\prod_{j\in J}\om_{\ell j})$, the tensor weight $\vec{\om}:=(\otimes_{i\in I} \om_{i1},\ldots,\otimes_{i\in I} \om_{im})$, and
 \begin{align*}
 [\vec{\ombf},\vec{\Ombf}]_{(\vec{\sbf},\vec{\rbf})}:=([\om_{ij},\Om_{ij}]_{(s_{ij},r_{ij})})_{(i,j)\in I\times J}.
 \end{align*}
  We write $[\vec{\ombf},\vec{\Ombf}]_{(\vec{\sbf},\vec{\rbf})}<\infty$ if $[\om_{ij},\Om_{ij}]_{(s_{ij},r_{ij})}<\infty$ for all $(i,j)\in I\times J$.
  Observe that $[\vec{\ombf},\vec{\Ombf}]_{(\vec{\sbf},\vec{\rbf})}<\infty$ implies $[\ombf,\Ombf]_{(\sbf,\rbf)}<\infty$ by virtue of Proposition \ref{mw_weight_prop}\ref{mw_weight_prop_iii}.
  Moreover, for $\vec{s}_0\in (0,\infty]^m$ and $\vec{\sbf}\in (0,\infty]^{\ell\times m}$, we introduce the shifted quantity
  \begin{align}\label{js241}
  \vec{\sbf}_0:=(s_{(i-1)j})_{(i,j)\in I\times J}.
  \end{align}
For a quasi-Banach space $X$ and a vector of exponents $\pbf:=(p_1,\ldots,p_\ell)\in (0,\infty]^\ell$, define the mixed-norm space $\LR{\pbf}(\Omega;X)$ as all strongly measurable $f:\Omega:=\prod_{i\in I} \Omega_i\to X$ such that
  \begin{align*}
  \|f\|_{\LR{\pbf}(\Omega;X)}:=\left(\int_{\Omega_\ell}\cdots \left(\int_{\Omega_1} \|f\|_{X}^{p_1} \mathrm{d}\mu_1\right)^{\frac{p_{2}}{p_1}}\cdots \mathrm{d}\mu_\ell\right)^{\frac1{p_\ell}}<\infty,
  \end{align*}
  with the usual modifications if $p_i=\infty$ for some $i\in I$.
  Given a weight $\om$, we define the weighted mixed-norm space $\LR{\pbf}_{\om}(\Omega)$ as all $f\in \LR{0}(\Omega,X)$ with
  \begin{align*}
   \|f\|_{\LR{\pbf}_{\om}(\Omega;X)}:=\|f\om\|_{\LR{\pbf}(\Omega;X)}<\infty.
  \end{align*}
  Often we consider a tensor weight $\om$ defined as the tensor weight of a vector of weights $\ombf$.
  In that case we write $\LR{\pbf}_{\ombf}:=\LR{\pbf}_\om$ to emphasize this point.
  For $\vec{\pbf}\in (0,\infty]^{\ell\times m}$ and a vector-tuple of weights $\vec{\ombf}$ we write $\LR{\vec{\pbf}}_{\vec{\ombf}}(\Omega):=\prod_{j\in J} \LR{\pbf_j}_{\ombf_j}(\Omega)$.

\medskip

  For any $\Lambda\subseteq I$ we write $\Omega_{\Lambda}:=\prod_{i\in\Lambda} \Omega_i$ and equip it with the corresponding product measure.
  We are particularly interested in the sets $\Lambda_i:=\set{i,\ldots,\ell}$ for $i\in I$.
  We will henceforth assume that $\Omega$ fulfills the following assumption.
  \begin{assumption}\label{ass_2}
  Each $\Omega_\Lambda$ comes with a basis of sets $\calu_\Lambda$, and $\calu_{\Lambda}\subseteq \calu_{\Lambda'}\times \calu_{\Lambda''}$ if $\Lambda$ is the disjoint union of $\Lambda'$ and $\Lambda''$.
  \end{assumption}
    For weights $\om'$ on $\Omega_{\Lambda'}$ and $\om''$ on $\Omega_{\Lambda''}$ we thus have $[\om'\otimes\om'']_{(s,r)}^{\calu_\Lambda}\le [\om']_{(s,r)}^{\calu_{\Lambda'}}[\om'']_{(s,r)}^{\calu_{\Lambda''}}$ for all $s,r\in (0,\infty]$.
  We usually drop the dependence on the bases in the notation $[\om]_{(s,r)}^{\calu_\Lambda}$.
In the non-mixed norm setting $\ell=1$, Definition \ref{js212} provided a natural framework for introducing weight classes and rescaled parameters.
To adapt this approach for the mixed-norm context, we now extend the definition.
\begin{definition}\label{js212a}
Let $\frac1{\gambf}\in\R^\ell$, $s_0,r_0\in (0,\infty]$, and let $\ombf$ and $\Ombf$ be tuples of weights.
Then we introduce
\begin{align*}
(\tbf_0,\tbf,\ombf_{\tbf},\Ombf_{\tbf},W)
\end{align*}
in the following way: Define $\sbf,\rbf\in (0,\infty]^\ell$ recursively via $\frac1{s_i}-\frac1{s_{i-1}}=\frac1{r_{i-1}}-\frac1{r_i}=\frac1{\gamma_i}$ (with the shifted notation $\frac1{\sbf}-\frac1{\sbf_0}=\frac1{\rbf_0}-\frac1{\rbf}=\frac1{\gambf}$) and set $\frac1\alpha:=\frac1{s_0}+\frac1{r_0}$.
Moreover, set
\begin{align}\label{js7e1a}
\begin{cases}
 t_{0i}:=\frac{s_{i-1}}{\alpha}, \quad t_i:=\frac{s_i}{\alpha}, \quad (\om_{t_i},\Om_{t_i}):=(\otimes_{k=i}^\ell \om_k^\alpha,\otimes_{k=i}^\ell \Om_k^\alpha) & \text{if } \frac1{\gamma_i}>0,\\
 t_{0i}:=1, \quad t_i:=1, & \text{if } \frac1{\gamma_i}=0,\\
 t_{0i}:=\frac{r_{i-1}}{\alpha}, \quad t_i:=\frac{r_i}{\alpha}, \quad (\om_{t_i},\Om_{t_i}):=(\otimes_{k=i}^\ell \Om_k^{-\alpha},\otimes_{k=i}^\ell\om_k^{-\alpha}) & \text{if } \frac1{\gamma_i}<0.
\end{cases}
\end{align}
and (recall $\om:=\otimes_{i\in I} \om_i$, $\Om:=\otimes_{i\in I} \Om_i$)
\begin{align}\label{js204a}
W:=
\begin{cases}
 \set{(\om,\Om)} & \text{if } \frac1{\gambf}=\mathbf{0}, \\
 \setc{(\tilde\om,\tilde\Om)}{\tilde\om\tilde\Om^{-1}=\om\Om^{-1} \text{ and } [\tilde\om,\tilde\Om]_{(s_0,r_0)}<\infty} & \text{else.} 
\end{cases}
\end{align}
For the case $\frac1{\gamma_i}=0$ we do not introduce $(\om_{t_i},\Om_{t_i})$, but use the following convention: $\calm$ is considered to be bounded from $\LR{t_i}_{\om_{t_i}}(\Omega_{\Lambda_i})$ to $\LR{t_i}_{\Om_{t_i}}(\Omega_{\Lambda_i})$.
\end{definition}
\begin{lemma}\label{js210}
Let $p_0,s_0,r_0\in (0,\infty]$ and $\pbf,\sbf,\rbf\in (0,\infty]^\ell$ be such that (with the notation in \eqref{js241})
\begin{align*}
\frac1{\pbf}-\frac1{\pbf_0}=\frac1{\sbf}-\frac1{\sbf_0}=\frac1{\rbf_0}-\frac1{\rbf}=:\frac1{\gambf}.
\end{align*}
Let $\ombf$ and $\Ombf$ be tuples of weights with $[\ombf,\Ombf]_{(\sbf,\rbf)}<\infty$.
Define
\begin{align*}
 (\tbf_0, \tbf, \ombf_{\tbf},\Ombf_{\tbf},W) \text{ via Definition \ref{js212a}}.
\end{align*}
 If $\calm_{\calu_{\Lambda_i}}$ is bounded from $\LR{t_{i}}_{\Om_{t_{i}}}(\Omega_{\Lambda_i})$ to $\LR{t_{i}}_{\om_{t_{i}}}(\Omega_{\Lambda_i})$ for all $i\in I$,
 then it holds
\begin{align*}
\LR{\pbf}_{\Ombf}(\Omega)\subseteq \bigcup_{(\om_0,\Om_0)\in W} \LR{p_0}_{\Om_0}(\Omega).
\end{align*}
\end{lemma}
\begin{proof}
 For $\ell=1$, the assertion is simply Corollary \ref{js201}.
 Let the result be true for all $\ell'<\ell$, and consider $f\in \LR{\pbf}_{\Ombf}(\Omega)$.
 Define $\tilde f\in \LR{p_\ell}_{\Om_\ell}(\Omega_\ell)$ by $\tilde f(x_\ell):=\|f(\cdot,x_\ell)\|_{\LR{\tilde{\pbf}}_{\tilde{\Ombf}}(\tilde\Omega)}$, where $\tilde{\pbf}=(p_1,\ldots,p_{\ell-1})$, $\tilde{\ombf}=(\om_1,\ldots,\om_{\ell-1})$, and $\tilde\Omega:=\prod_{i=1}^{\ell-1}\Omega_i$.
 The induction hypothesis yields weights $\mu$ and $\nu$ on $\Omega_\ell$ with $(\mu,\nu)=(\om_{\ell},\Om_{\ell})$ if $\frac1{\gamma_\ell}=0$ and $\mu\nu^{-1}=\om_{\ell}\Om_{\ell}^{-1}$ otherwise, such that $\tilde f\in \LR{p_{\ell-1}}_{\nu}(\Omega_\ell)$, i.e. $f\in \LR{\tilde{\pbf}}_{(\Om_1,\ldots,\Om_{\ell-1}\otimes \nu)}(\Omega)$.
 Thus, again by the induction hypothesis, there are weights $\om_0$ and $\Om_0$ on $\Omega$ with $(\om_0,\Om_0)=(\otimes_{i\in I} \om_i,\otimes_{i\in I} \Om_i)$ if $\frac1{\gambf}=0$ and $\frac{\om_0}{\Om_0}=(\frac{\om_1}{\Om_1})\otimes\cdots\otimes(\frac{\om_{\ell-1}\otimes\mu}{\Om_{\ell-1}\otimes\nu})$ otherwise, such that $f\in \LR{p_0}_{\Om_0}(\Omega)$.
 Since $(\frac{\om_{\ell-1}\otimes\mu}{\Om_{\ell-1}\otimes\nu})=(\frac{\om_{\ell-1}}{\Om_{\ell-1}})\otimes (\frac{\om_{\ell}}{\Om_{\ell}})$ by the definition of $\mu$ and $\nu$, the result follows.
\end{proof}
\begin{theorem}\label{mw_extra_pol_prop}
Let $q_0\in (0,\infty]$, $\vec{p}_0,\vec{s}_0,\vec{r}_0\in (0,\infty]^m$, $\qbf\in (0,\infty]^\ell$ and $\vec{\pbf},\vec{\sbf},\vec{\rbf}\in (0,\infty]^{\ell\times m}$ be such that (with the notation in \eqref{js241})
\begin{align}\label{js011}
 \frac{1}{\vec{\pbf}}-\frac{1}{\vec{\pbf}_0}=\frac{1}{\vec{\sbf}}-\frac{1}{\vec{\sbf}_0}=\frac{1}{\vec{\rbf}_0}-\frac{1}{\vec{\rbf}}=:\frac1{\vec{\gambf}}, \quad \frac1{\qbf}-\frac1{\qbf_0}=\frac1{\gambf}.
\end{align}
Let $\vec{\ombf}$ and $\vec{\Ombf}$ be vector-tuples of weights with $[\vec{\ombf},\vec{\Ombf}]_{(\vec{\sbf},\vec{\rbf})}<\infty$, and define for $j\in J$
\begin{align*}
(\tbf_{0j},\tbf_j,\ombf_{\tbf_j},\Ombf_{\tbf_j},W_j) \text{ via Definition \ref{js212a}}.
\end{align*}
We assume the following:
\begin{enumerate}
\item\label{mw_extra_pol_prop_iv} For all $(i,j)\in I\times J$ there is an increasing function $\psi_{ij}:\R\to\R$ with
\begin{align*}
\|\calm_{\calu_{\Lambda_i}}\|_{\LR{t_{ij}}_{\Om_{t_{ij}}}(\Omega_{\Lambda_i})\to \LR{t_{ij}}_{\om_{t_{ij}}}(\Omega_{\Lambda_i})}\le \psi_{ij}([\om_{ij},\Om_{ij}]_{(s_{ij},r_{ij})}),
\end{align*}
where $\Lambda_i:=\set{i,\ldots,\ell}$.
\item\label{mw_extra_pol_prop_iii} For $j\in J$, there is a set $V_j$, a map $S_j:V_j\to \LR{0}(\Omega)$, a map
 \begin{align*}
 T: \prod_{j\in J} \bigcup_{(\om_{0j},\Om_{0j})\in W_j} S_j^{-1}(\LR{p_{0j}}_{\Om_{0j}}(\Omega)) \to \LR{0}(\Omega),
 \end{align*}
 and a function $\phi:\R^m\to \R$ increasing in each component such that for all $(\vec{\om}_0, \vec{\Om}_0)\in \vec{W}:=W_1\times\cdots\times W_m$ and $\vec{f}\in \prod_{j\in J} S_j^{-1}(\LR{\vec{p}_0}_{\vec{\Om}_0}(\Omega))$ it holds
\begin{align}\label{js010}
 \|T\vec{f}\|_{\LR{q_0}_{\om_0}(\Omega)}\leq \phi([\vec{\om}_0,\vec{\Om}_0]_{(\vec{s}_0,\vec{r}_0)})\prod_{j\in J}\|S_j f_j\|_{\LR{p_{0j}}_{\Om_{0j}}(\Omega)},
\end{align}
\end{enumerate}
Then there is a function $\phi_{\vec{s_0},\vec{r_0},\vec{\gambf}}:\R^{\ell\times m}\to \R$ increasing in each component such that
$T\vec{f}$ is well-defined for all $\vec{f}\in \prod_{j\in J} S_j^{-1}(\LR{\pbf_j}_{\Ombf_j}(\Omega))$ and
\begin{align}\label{js012ab}
 \|T\vec{f}\|_{\LR{\qbf}_{\ombf}(\Omega)}\leq \phi_{\vec{s_0},\vec{r_0},\vec{\gambf}}([\vec{\ombf},\vec{\Ombf}]_{\vec{\sbf},\vec{\rbf}}) \prod_{j\in J} \|S_j f_j\|_{\LR{\pbf_j}_{\Ombf_j}(\Omega)}.
\end{align}
\end{theorem}
\begin{proof}
We establish the conclusion by induction $\ell\in\N$ in a slightly stronger form:
There are $\phi_j=\phi_{j,\vec{s_0},\vec{r_0},\vec{\gambf}}:(1,\infty)\times\R^{\ell}\to \R$ increasing in each component with the following property:
$T\vec{f}$ is well-defined for all $\vec{f}\in \prod_{j\in J} S_j^{-1}(\LR{\pbf_j}_{\Ombf_j}(\Omega))$, and for all $\kappa\in (1,\infty)$ it holds
\begin{align}\label{js012}
 \|T\vec{f}\|_{\LR{\qbf}_{\ombf}(\Omega)}\leq \kappa^{\ell/\alpha}\phi((\phi_j(\kappa',[\ombf_j,\Ombf_j]_{\sbf_j,\rbf_j}))_{j\in J}) \prod_{j\in J} \|S_j f_j\|_{\LR{\pbf_j}_{\Ombf_j}(\Omega)}.
\end{align}
Let $\vec{f}\in \prod_{j=1}^m S_j^{-1}(\LR{\pbf_j}_{\Ombf_j}(\Omega))$.
Then $T\vec{f}$ is well-defined by Lemma \ref{js210}.
Fix $\kappa\in (1,\infty)$.
Observe that the case $\ell=1$ holds by Theorem \ref{mw_extra_pol_thm_main} and the assumption in \ref{mw_extra_pol_prop_iv}.
Suppose that the assertion is true for all $\ell'<\ell$.
Let $\vec{\mu},\vec{\nu}$ be weight tuples on $\Omega_{\Lambda_2}=\prod_{i=2}^{\ell} \Omega_i$ with $(\mu_j,\nu_j)\in \widetilde{W}_{j}:=W(s_{1j},r_{1j}, \otimes_{i=2}^\ell \om_{ij}\Om_{ij}^{-1}) $.
Then $[\vec{\om}_1\otimes\vec{\mu},\vec{\Om}_1\otimes \vec{\nu}]_{(\vec{s}_1,\vec{r}_1)}\le[\vec{\om}_1,\vec{\Om}_1]_{(\vec{s}_1,\vec{r}_1)}[\vec{\mu},\vec{\nu}]_{(\vec{s}_1,\vec{r}_1)}<\infty$.
Estimate \eqref{js010} is valid by assumption for all $(\vec{\om}_0,\vec{\Om}_0)$ with $(\om_{0j},\Om_{0j})\in W(s_{0j},r_{0j},\frac{\om_{1j}\otimes \mu_{j}}{\Om_{1j}\otimes\nu_{j}})\subseteq W_j$ for all $j\in J$ and all $\vec{g}\in \prod_{j\in J} S_j^{-1}(\LR{\vec{p}_0}_{\vec{\Om}_0}(\Omega))$, and serves as input \eqref{js010a} to Theorem \ref{mw_extra_pol_thm_main}.
This yields that $T\vec{g}$ is well-defined for all $\vec{g}\in \prod_{j\in J} S_j^{-1}(L^{p_{1j}}_{\Om_{1j}\otimes\mu_j}(\Omega))$ and
\begin{align}\label{js13}
 \|T\vec{g}\|_{\LR{q_1}_{\om_1\otimes\mu}}\le \kappa^{\frac1{\alpha}}\phi((C_{1 j}(\kappa',[\om_{1j},\Om_{1j}]_{(s_{1 j},r_{1 j})},[\mu_j,\nu_j]_{(s_{1 j},r_{1 j})}))_{j\in J}) \prod_{j\in J}\|S_j g_j\|_{\LR{p_{1 j}}_{\Om_{1 j}\otimes\nu_j}},
\end{align}
where $C_{1 j}(\rho,x_j,y_j):=(\rho \psi_{1j}(x_jy_j))^{t_{1j}/|\gamma_{1j}|}(x_jy_j)^{t_{1j}/t_{01j}}$.
Let $\widetilde{S}_j:V_j\to \LR{0}(\Omega_{\Lambda_2})$ be given by $\widetilde{S}_j g_j(x'):=\|S_j g_j(x',\cdot)\|_{\LR{p_{1 j}}_{\Om_{1 j}}(\Omega_1)}$ and
 \begin{align*}
 \widetilde T: \prod_{j=1}^m \bigcup_{(\om_{0j},\Om_{0j})\in \widetilde{W}_j}\widetilde{S}_j^{-1}(\LR{p_{0j}}_{\Om_{0j}}(\Omega_{\Lambda_2})) \to \LR{0}(\Omega_{\Lambda_2}), \qquad
 \widetilde T\vec{g}(x'):=\|T\vec{g}(x',\cdot)\|_{\LR{q_1}_{\om_1}(\Omega_1)},
 \end{align*}
so that
 \eqref{js13} reads
\begin{align*}
\|\widetilde T\vec g\|_{\LR{q_1}_{\mu}(\Omega_{\Lambda_2})}\le \tilde\phi([\vec{\mu},\vec{\nu}]_{(\vec{s}_1,\vec{r}_1)})\prod_{j\in J}\|\widetilde{S}_j g_j\|_{\LR{p_{1 j}}_{\nu_j}(\Omega_{\Lambda_2})},
\end{align*}
where $\tilde\phi:\R^m\to \R$, $\tilde\phi(y_1,\ldots,y_m):=\kappa^{\frac1{\alpha}}\phi((C_{1 j}(\kappa',[\om_{1j},\Om_{1j}]_{(s_{1 j},r_{1 j})},y_j))_j)$.
Since $\frac1{\alpha_j}=\frac1{s_{0j}}+\frac1{r_{0j}}=\frac1{s_{1 j}}+\frac1{r_{1 j}}$, we may use the induction hypothesis for $(q_0,\vec{p}_0,\vec{s}_0,\vec{r}_0)\mapsto (q_1,\vec{p}_1,\vec{s}_1,\vec{r}_1)$, $\phi\mapsto \tilde\phi$, $W_j\mapsto \widetilde{W}_j$, $S\mapsto \widetilde S$ and $T\mapsto \widetilde T$ to obtain for all $j\in J$ functions $\tilde{\phi}_j:(1,\infty)\times \R^{\ell-1}\to \R$ with
\begin{align}\label{js012a}
 \|\widetilde{T}\vec{f}\|_{\LR{\tilde{\qbf}}_{\tilde{\ombf}}(\Omega_{\Lambda_2})}\leq \kappa^{(\ell-1)/\alpha}\tilde\phi((\tilde{\phi}_j(\kappa',[\tilde{\ombf}_j,\tilde{\Ombf}_j]_{\tilde{\sbf}_j,\tilde{\rbf}_j}))_{j\in J}) \prod_{j\in J} \|\widetilde{S}_j f_j\|_{\LR{\tilde{\pbf}_j}_{\tilde{\Ombf}_j}(\Omega_{\Lambda_2})},
\end{align}
where $\tilde{\qbf}:=(q_2,\ldots,q_\ell)$, $\tilde{\ombf}_j=(\om_{2j},\ldots,\om_{\ell j})$ and similarly for $\tilde{\Ombf}_j$, $\tilde{\pbf}_j$, $\tilde{\sbf}_j$, and $\tilde{\rbf}_j$.
Defining $\phi_j(\rho, y_{1j},\ldots,y_{\ell j}):=C_{1j}(\rho,y_1,\tilde{\phi}_j(\rho,y_{2},\ldots,y_{\ell}))$, estimate \eqref{js012a} turns into \eqref{js012}.
\end{proof}

\begin{remark}\label{mw_extra_pol_rm}
 \begin{enumerate}
 \item The proof of Theorem \ref{mw_extra_pol_prop} allows to track the exact definition of $\phi_{\vec{s_0},\vec{r_0},\vec{\gambf}}$ recursively.
 In particular, assume
 that there is $c:(1,\infty]\to [1,\infty)$ such that for all $p\in (1,\infty]$, $i\in I$, and weights $\om$ on $\Omega_{\Lambda_i}$ with $[\om]_p^{\calu_{\Lambda_i}}<\infty$, we have that $\calm_{\calu_{\Lambda_i}}$ is bounded in $\LR{p}_{\om}(\Omega_{\Lambda_i})$ with
 \begin{align*}
  \|\calm_{\calu_{\Lambda_i}}\|_{\LR{p}_{\om}(\Omega_{\Lambda_i})\to \LR{p}_{\om}(\Omega_{\Lambda_i})}\le c(p)[\om]_p^{p'}.
 \end{align*}
 Then if $\vec{\sbf},\vec{\rbf}\in (0,\infty)^{\ell\times m}$, $\vec{\ombf}=\vec{\Ombf}$, and $\kappa\in (1,\infty)$, an iteration of the bound \eqref{js202_e1} in Remark \ref{js202} shows that we may choose $\phi_{\vec{s_0},\vec{r_0},\vec{\gambf}}([\vec{\ombf},\vec{\Ombf}]_{\vec{\sbf},\vec{\rbf}}):=\kappa^{\ell/\alpha}\phi(\vec{C})$ with
 \begin{align*}
C_{j}:=&
 \prod_{i\in I} (\kappa'c(t_{ij}))^{\frac{t_{ij}b_{(i-1)j}}{|\gamma_{ij}|}}[\om_{ij}]_{(s_{ij},r_{ij})}^{b_{ij}},
\end{align*}
where $b_{ij}=\prod_{k=1}^i t_{kj}'/t_{0kj}'$ for $i\in I$.
  \item
  Theorem \ref{mw_extra_pol_prop} also contains a weak-type extrapolation result by an argument of Grafakos and Martell \cite{GrM04}:
  One may exchange $ \|T\vec{f}\|_{\LR{q_0}_{\om_0}(\Omega)}$ by $ \|T\vec{f}\|_{\LR{q_0,\infty}_{\om_0}(\Omega)}$ in \eqref{js010} if one also exchanges $\|T\vec{f}\|_{\LR{\qbf}_{\ombf}(\Omega)}$ by $\|T\vec{f}\|_{\LR{\qbf,\infty}_{\ombf}(\Omega)}$ in \eqref{js012ab}.
    Here, the weak mixed-norm space $\LR{\pbf,\infty}(\Omega)$ is given as all $f\in \LR{0}(\Omega)$ such that
  \begin{align*}
  \|f\|_{\LR{\pbf,\infty}(\Omega)}:=\sup_{\lambda>0} \lambda \|\1_{E_\lambda}\|_{\LR{\pbf}(\Omega)}<\infty,
  \end{align*}
  where $E_\lambda:=\setc{x\in \Omega}{|f|>\lambda}$, and $f\in \LR{\pbf,\infty}_{\ombf}(\Omega)$ if $f\om\in \LR{\pbf,\infty}(\Omega)$, where $\om$ is the tensor weight induced by $\ombf$.
  To see the weak-type extrapolation, consider for $\lambda>0$ the set $E_{\lambda}:=\setc{x\in\Omega}{|T\vec{f}(x)|>\lambda}$ and define $T_\lambda \vec{f}:=\lambda\1_{E_\lambda}$.
  Applying Theorem \ref{mw_extra_pol_prop} with $T\mapsto T_\lambda$, the assertion is obtained in the same way as in \cite[Theorem 6.1]{GrM04}.
 \end{enumerate}
\end{remark}
Finally, we obtain a vector-valued extrapolation result.
\begin{theorem}\label{rb_extra_pol_thm_main}
 Assume that the proviso of Theorem \ref{mw_extra_pol_prop} is fulfilled,
 and introduce the exponent $\frac1{u_0}:=\min\set{\frac1{q_0},\frac{1}{p_0}}$ and $\vec{u}_{0}:=\frac{u_0}{p_0}\vec{p}_{0}$, where we recall $\frac1{p_0}:=\sum_{j\in J}\frac1{p_{0j}}$.
 Let
 \begin{align*}
 \mathcal{T}\subseteq \set{T: \prod_{j\in J} \bigcup_{(\om_{0j},\Om_{0j})\in W_j} S_j^{-1}(\LR{p_{0j}}_{\Om_{0j}}(\Omega)) \to \LR{0}(\Omega)}.
 \end{align*}
 If there is a function $\phi:\R^m\to \R$ increasing in each component such that for all $T\in \calt$, $(\vec{\om}_0, \vec{\Om}_0)\in \vec{W}:=W_1\times\cdots\times W_m$, and $\vec{f}\in \prod_{j\in J} S_j^{-1}(\LR{\vec{p}_0}_{\vec{\Om}_0}(\Omega))$ it holds
 \begin{align}\label{js010ab}
 \|T\vec{f}\|_{\LR{q_0}_{\om_0}(\Omega)}\leq \phi([\vec{\om}_0,\vec{\Om}_0]_{(\vec{s}_0,\vec{r}_0)})\prod_{j\in J}\|S_j f_j\|_{\LR{p_{0j}}_{\Om_{0j}}(\Omega)},
 \end{align}
 then for all $T\in\mathcal{T}$ and all $\vec{f}\in \prod_{j\in J} S_j^{-1}(\LR{\pbf_j}_{\Ombf_j}(\Omega))$, $T\vec{f}$ is well defined, and
 \begin{align*}
   \|(T_k(f_{1k},\ldots,f_{jk}))_{k\in\N}\|_{\LR{\qbf}_{\ombf}(\Omega;\ell^{u_0})}\leq \phi_{\vec{s_0},\vec{r_0},\vec{\gambf}}([\vec{\ombf},\vec{\Ombf}]_{\vec{\sbf},\vec{\rbf}})\prod_{j\in J} \|(f_{jk})_{k\in\N}\|_{\LR{\pbf_j}_{\Ombf_j}(\Omega;\ell^{u_{0j}})},
 \end{align*}
 for all $(T_k)_{k\in \N}\subseteq \mathcal{T}$ and $(f_{jk})_{k\in\N}\subseteq V_j$ with $(S_jf_{jk})_{k\in\N}\in \LR{\pbf_j}_{\Ombf_j}(\Omega;\ell^{u_{0j}})$ for all $j\in J$.
 The function $\phi_{\vec{s_0},\vec{r_0},\vec{\gambf}}$ can be taken identical to the one in Theorem \ref{mw_extra_pol_prop}.
\end{theorem}
\begin{proof}
For
\begin{align*}
\widetilde{V}_j:=
\setc{(f_{jk})_{k\in\N}\subseteq V_j}{(S_j f_{jk})_{k\in\N}\in \bigcup_{(\om_{0j},\Om_{0j})\in W_j} \LR{p_{0j}}_{\Om_{0j}}(\Omega; \ell^{u_{0j}})}
\end{align*}
define the maps $\widetilde{S}_j:\widetilde{V}_j\to \LR{0}(\Omega)$ and $\widetilde{T}:\prod_{j=J} \widetilde{V}_j 
\to \LR{0}(\Omega)$ by
\begin{align*}
\widetilde{S}_j f_j:=\|(S_j f_{jk})_{k}\|_{\ell^{u_{0j}}}, \quad \widetilde{T}\vec{f}:=\|(T_k\vec{f}_k)_k\|_{\ell^{u_0}}.
\end{align*}
We claim that $\widetilde{S}_j^{-1}(\LR{\pbf_j}_{\Ombf_j}(\Omega))=\setc{(f_{jk})_{k\in\N}\subseteq V_j}{(S_jf_{jk})_{k\in\N}\in \LR{\pbf_j}_{\Ombf_j}(\Omega;\ell^{u_{0j}})}$ for all $j\in J$.
The inclusion ``$\subseteq$'' follows directly from the definition of $\widetilde{S}_j$.
For ``$\supseteq$'' let $(f_{jk})_{k\in\N}$ be in the set on the right-hand side and consider the function 
$h:=\|(S_j f_{jk})_{k\in\N}\|_{\ell^{u_{0j}}}\in \LR{\pbf_j}_{\Ombf_j}(\Omega)$.
Then $h\in \bigcup_{(\om_{0j},\Om_{0j})\in W_j} \LR{p_{0j}}_{\Om_{0j}}(\Omega)$ by Lemma \ref{js210}, that is $(f_{jk})_{k\in\N}\subseteq \widetilde{V}_j$.
Hence we may apply $\widetilde{S}_j$ to $(f_{jk})_{k\in\N}$ to see that $\widetilde{S}_j(f_{jk})_{k\in\N}=h\in \LR{\pbf_j}_{\Ombf_j}(\Omega)$, i.e. $(f_{jk})_{k\in\N}\in \widetilde{S}_j^{-1}(\LR{\pbf_j}_{\Ombf_j}(\Omega))$.

\medskip

For all weight-tuples $(\vec{\om}_0,\vec{\Om}_0)\in \vec{W}$ and all $\vec{f}\in \prod_{j=1}^m \widetilde{S}^{-1}(\LR{p_{0j}}_{\om_{0j}}(\Omega))$ we have both
\begin{align*}
\|\widetilde{T}\vec{f}\|_{\LR{q_0}_{\vec{\om}_0}(\Omega)}\le \|(T_k\vec{f}_k)_{k}\|_{\ell^{u_0}(\LR{q_0}_{\vec{\om}_0}(\Omega))}
\end{align*}
and for all $j\in J$
\begin{align*}
\|(S_jf_{jk})_k\|_{\ell^{u_{0j}}(\LR{p_{0j}}_{\Om_{0j}}(\Omega))}\le \|\widetilde{S}_jf_j\|_{\LR{p_{0j}}_{\Om_{0j}}(\Omega)}.
\end{align*}
Indeed, this follows by Minkowski's inequality for the first estimate and by Fubini for the second estimate, or vice versa depending on $q_0=u_0$ or $q_0\ne u_0$ (observe that $q_0\ne u_0$ implies $p_{0j}= u_{0j}$ for all $j\in J$).
Thus there holds by \eqref{js010ab} and H\"older's inequality
\begin{align*}
\|\widetilde{T}\vec{f}&\|_{\LR{q_0}_{\vec{\om}_0}(\Omega)}\le \|(T_k\vec{f}_k)_{k}\|_{\ell^{u_0}(\LR{q_0}_{\vec{\om}_0}(\Omega))}\le \phi([{\vec{\om}_0}]_{(\vec{s}_0,\vec{r}_0)})\left\|\prod_{j=1}^m \|(S_jf_{jk})_k\|_{\LR{p_{0j}}_{\Om_{0j}}(\Omega)}\right\|_{\ell^{u_0}}\\
&\le \phi([{\vec{\om}_0}]_{(\vec{s}_0,\vec{r}_0)})\prod_{j=1}^m \|(S_jf_{jk})_k\|_{\ell^{u_{0j}}(\LR{p_{0j}}_{\Om_{0j}}(\Omega))}\le \phi([{\vec{\om}_0}]_{(\vec{s}_0,\vec{r}_0)})\prod_{j=1}^m \|\widetilde{S}_jf_j\|_{\LR{p_{0j}}_{\Om_{0j}}(\Omega)}.
\end{align*}
Consequently, $\widetilde{T}\vec{f}$ is well-defined for all $\vec{f}\in \prod_{j=1}^\infty \widetilde{S}_j^{-1}(\LR{\pbf_j}_{\Ombf_j}(\Omega))$ by Theorem \ref{mw_extra_pol_prop} with
\begin{align*}
\|\widetilde{T}\vec{f}\|_{\LR{\qbf}_{\ombf}(\Omega)}\le \phi_{\vec{s}_0,\vec{r}_0,\vec{\sbf}}([\vec{\ombf}]_{(\vec{\sbf},\vec{\rbf})})\prod_{j=1}^m \|\widetilde{S}_jf_j\|_{\LR{\pbf_j}_{\Ombf_j}(\Omega)},
\end{align*}
which is the assertion.
\end{proof}
We finally obtain a result about $\calr$-boundedness of families of linear operators via extrapolation.
For the definition of $\calr$-boundedness and related concepts we refer to the monograph \cite{HNV17}.
We mention its pivotal role in the theory of maximal $\LR{p}$ regularity via Weis' theorem \cite{Wei01}.
\begin{corollary}
Let $q_0, p_0,s_0,r_0\in (0,\infty]$, $\qbf\in [1,\infty)^{\ell}$, $\pbf\in [1,\infty]^\ell$, and $\sbf,\rbf\in (0,\infty]^{\ell}$ be such that (with the notation in \eqref{js241})
\begin{align*}
 \frac1{\qbf}-\frac1{\qbf_0} = \frac{1}{\pbf}-\frac{1}{\pbf_0}=\frac{1}{\sbf}-\frac{1}{\sbf_0}=\frac{1}{\rbf_0}-\frac{1}{\rbf}=:\frac1{\gambf}.
\end{align*}
Let $\ombf$ and $\Ombf$ be tuples of weights with $[\ombf,\Ombf]_{(\sbf,\rbf)}<\infty$, and define
\begin{align*}
(\tbf_{0},\tbf,\ombf_{\tbf},\Ombf_{\tbf},W) \text{ via Definition \ref{js212a}}.
\end{align*}
We assume $\min\set{\frac1{q_0},\frac1{p_0}}=\frac12$ and the following:
\begin{enumerate}
\item\label{mw_extra_pol_prop_iv} For all $i\in I$ there is an increasing function $\psi_{i}:\R\to\R$ with
\begin{align*}
\|\calm_{\calu_{\Lambda_i}}\|_{\LR{t_{i}}_{\Om_{t_{i}}}(\Omega_{\Lambda_i})\to \LR{t_{i}}_{\om_{t_{i}}}(\Omega_{\Lambda_i})}\le \psi_{i}([\om_{i},\Om_{i}]_{(s_{i},r_{i})}),
\end{align*}
where $\Lambda_i:=\set{i,\ldots,\ell}$.
\item\label{mw_extra_pol_prop_iii} There is $\mathcal{T}\subseteq \set{T: \bigcup_{(\om_{0},\Om_{0})\in W} \LR{p_{0}}_{\Om_{0}}(\Omega) \to \LR{0}(\Omega) \text{ is linear}}$
 and a function $\phi:\R^m\to \R$ increasing in each component such that for all $(\om_0, \Om_0)\in W$, all $T\in \calt$ and all $f\in \LR{p_0}_{\Om_0}(\Omega)$ it holds
\begin{align*}
 \|Tf\|_{\LR{q_0}_{\om_0}(\Omega)}\leq \phi([\om_0,\Om_0]_{s_0,r_0)})\|f\|_{\LR{p_{0}}_{\Om_{0}}(\Omega)}.
\end{align*}
\end{enumerate}
Then there is a function $\phi_{s_0,r_0,\gambf}:\R^{\ell}\to \R$ increasing in each component such that
$Tf$ is well-defined for all $T\in \calt$ and $f\in  \LR{\pbf}_{\Ombf}(\Omega)$, and $\calt\subseteq \call(\LR{\pbf}_{\Ombf}(\Omega),\LR{\qbf}_{\ombf}(\Omega))$ is $\calr$-bounded with
\begin{align*}
 \calr(\calt)\le \phi_{s_0,r_0,\gambf}([\ombf,\Ombf]_{\sbf,\rbf}).
\end{align*}
\end{corollary}
\begin{proof}
By Theorem \ref{rb_extra_pol_thm_main} applied with $m=1$, $V:=L^0(\Omega)$ and $S:=\id$, we have
  \begin{align*}
   \|(T_k f_{k})_{k\in \N}\|_{\LR{\qbf}_{\ombf}(\Omega;\ell^{2})}\leq \phi_{\vec{s_0},\vec{r_0},\vec{\gambf}}([\vec{\ombf},\vec{\Ombf}]_{\vec{\sbf},\vec{\rbf}}) \|(f_{k})_{k\in\N}\|_{\LR{\pbf}_{\Ombf}(\Omega;\ell^{2})}
 \end{align*}
 for all $(T_k)_{k\in \N}\subseteq \mathcal{T}$ and $(f_{k})_{k\in\N}\subseteq \LR{\pbf}_{\ombf}(\Omega;\ell^{u_{2}})$.
 By our assumptions on $\qbf$ and $\pbf$, both $\LR{\pbf}_{\Ombf}(\Omega)$ and $\LR{\qbf}_{\ombf}(\Omega)$ are Banach spaces.
 Since $\LR{\qbf}_{\ombf}(\Omega)=\LR{\qbf}(\Omega, \otimes \om_i^{q_i}\dd\mu)$ has finite cotype with cotype-bound independent of the weight by \cite[Corollary 7.1.5]{HNV17}, the $\calr$-bound follows by \cite[Theorem 8.1.3(iii)]{HNV17}.
\end{proof}

\section{Weighted Transference Principle}

In this section we specialize to a locally abelian group $\Omega=\gr$ and to integrability exponents in the Banach space range.
Let $\mu$ be a nontrivial, regular measure on a $\gr$ with $\mu(K)<\infty$ for all compact sets $K\subseteq \gr$.
Then the space of continuous functions with compact support $\CRc{}(\gr)$ is dense in $\LR{p}(\gr)$ for all $p\in[1,\infty)$, see Appendix E.8 of \cite{Rud62}.

\medskip

We will work with a basis of sets which enables us to talk about a doubling property of the measure $\mu$ with respect to this basis.
We therefore make the following assumption, which is met by the groups $\R$, $\Z$, the torus $\mathbb{T}$ and finite products of these groups.
\begin{assumption}\label{mw_ass_ball}
 Suppose that $\gr$ is a locally compact abelian group equipped with a nontrivial and regular measure $\mu$, such that $\mu(K)<\infty$ for all compact $K\subseteq \gr$.
 Furthermore, assume that there is a local base of $0\in \gr$ consisting of relatively compact measurable neighbourhoods $U_k$, $k\in\Z$, symmetric in the sense $U_k=-U_k$, such that
 \begin{enumerate}
  \item\label{mw_ass_ball_i} $\bigcup\limits_{k\in\Z}U_k=\gr$,
  \item\label{mw_ass_ball_ii} $U_k\subseteq U_m$, if $k\leq m$,
  \item\label{mw_ass_ball_iii} there exist a positive constant $A$ and a non-decreasing mapping $\theta:\Z\to\Z$ such that for all $k\in\Z$ and all $x\in \gr$
  \begin{align*}
   k&<\theta(k), \\
   2U_k&\subseteq U_{\theta(k)}, \\
   \mu(x+U_{\theta(k)})&\leq A\mu(x+U_{k}).
  \end{align*}
  Observe that necessarily $A\geq 1$ because $U_k\subseteq U_{\theta(k)}$.
 \end{enumerate}
\end{assumption}
 
 \medskip

In \cite{Sau15a}, it was shown that for $p\in(1,\infty)$ weights with $[\om]_p<\infty$ are exactly those weights such that the maximal operator is bounded in $L_{\om}^{p}(\gr)$.
Here, we extend this result to the endpoint $p=\infty$.
\begin{proposition}\label{mw_muckenhoupt_main}
 Assume $p\in(1,\infty]$ and let $\om$ be a weight on $\gr$.
 Then $\mathcal{M}_{\gr}$ is bounded in $L_{\om}^{p}(\gr)$ if and only if $[\om]_p<\infty$.
 Moreover, there exists $\cA(p)\in \R$ such that the bound $\|\mathcal{M}_{\gr}\|_{\LR{p}_\om(\gr)\to\LR{p}_\om(\gr)}\le \cA(p)[\om]_p^{p'}$ is valid.
\end{proposition}
\begin{proof}
For $p\in (1,\infty)$, the proof can be found in Theorem 1.2 of \cite{Sau15a}, with the explicit bound proved in Theorem 1.3 of \cite{PaR19}, where we recall that $[\om^p]_{A_p}=[\om]_p^p$.
For $p=\infty$, Proposition \ref{mw_weight_prop}\ref{mw_weight_prop_v} gives the result with $\cA(\infty)=1$.
\end{proof}

\begin{definition}\label{js104}
 Let ${\gr}$ be a locally compact abelian group with Haar measure $\mu$ subject to Assumption \ref{mw_ass_ball}.
 We say that $f\in \Loq{}{\infty}({\gr})$ is rapidly decaying and write $f\in \Loq{\mathrm{dec}}{\infty}(\gr)$ if
 \begin{align}\label{mw_rap_dec_est}
  (\forall n\in \N)(\exists c_n>0)(\forall k\in\N) \quad \|A^{nk}f\|_{\Loq{}{\infty}({\gr}\setminus U_{\theta^k(0)})}\leq c_n.
 \end{align}
 For each $n\in \N$ we define $\rho_n:\LR{\infty}(\gr)\to \R$ via
 \begin{align*}
 \rho_n(f):=\max\set{\|f\|_{\LR{\infty}(U_0)},\sup_{k\in\N}\|A^{nk} f\|_{\LR{\infty}(\gr\setminus U_{\theta^k(0)})}}.
 \end{align*}
\end{definition}
 It follows immediately from the definition that for $f\in \Loq{\mathrm{dec}}{\infty}({\gr})$ it holds ${f(\cdot-y)\in \Loq{\mathrm{dec}}{\infty}({\gr})}$ for all $y\in {\gr}$, and $\Loq{\mathrm{dec}}{\infty}({\gr})\Loq{}{\infty}({\gr})\subseteq \Loq{\mathrm{dec}}{\infty}({\gr})$, so that $\Loq{\mathrm{dec}}{\infty}({\gr})$ is closed under pointwise multiplication.
\begin{lemma}\label{js41}
 We collect the following properties of $\Loq{\mathrm{dec}}{\infty}({\gr})$.
 \begin{enumerate}
  \item For all $n\in\N$, $p\in (\frac1n,\infty]$, and weights $\om$ with $[\om]_{p,r}<\infty$, $r\in [(np)'/n,\infty]$, there holds
  \begin{align*}
  \|f\|_{\LR{p}_{\om}(\gr)}\le A[\om]_{p,r}^{(np)'} \|\om\|_{\LR{p}(U_0)} \rho_n(f),
  \end{align*}
   so that $\Loq{\mathrm{dec}}{\infty}({\gr})\subseteq \LR{p}_{\om}({\gr})$.
  A particular instance is $n=2$, $p\in [1,\infty)$ and $[\om]_p<\infty$.
  Furthermore for each $n\in\N$ there is a constant $c_n>0$ such that for all $k\in\N$ it holds
   \begin{align*}
    \|A^{nk} f\|_{\LR{1}({\gr}\setminus U_{\theta^k(0)})}\leq c_n.
   \end{align*}
  \item $\Loq{\mathrm{dec}}{\infty}({\gr})$ constitutes a convolution algebra, \emph{i.e.}, for $f,g\in \Loq{\mathrm{dec}}{\infty}(\gr)$ it holds true that ${f\ast g\in \Loq{\mathrm{dec}}{\infty}(\gr)}$, where
  \begin{align*}
   (f\ast g)(x):=\myint{\gr}{f(x-y)g(y)}{\mu(y)}.
  \end{align*}
 \end{enumerate}
\end{lemma}
\begin{proof}
 All points except the inclusion $\Loq{\mathrm{dec}}{\infty}({\gr})\subseteq \Loq{\om}{q}({\gr})$ can be found in \cite{Osb75}.
 In order to prove the inclusion, fix $q\in [1,\infty]$, a weight $\om$ with $[\om]_q<\infty$, and $f\in \Loq{\mathrm{dec}}{\infty}({\gr})$.
 Define $P:\gr\to \R$ via
 \begin{align*}
 P(x):=
 \begin{cases}
 A^{-nk} & \tif x\in U_{\theta^k(0)}\setminus U_{\theta^{k-1}(0)}, k\ge 1,\\
 1 & \tif x\in U_0.
 \end{cases}
 \end{align*}
 We observe that the definition of $\rho_n$ ensures that
 \begin{align*}
 \|f\cdot P^{-1}\|_{\LR{\infty}(\gr)}\le \rho_n(f).
 \end{align*}
 We now argue that there is $C>0$ such that $\sqrt[n]{P(x)}\le A \calm_\gr \psi_0(x)$ for all $x\in \gr$, where $\psi_0\in \LR{1}(\gr)\cap \LR{\infty}(\gr)$ is defined as $\psi_0:=\chi_{U_0}$.
 This is clear for $x\in U_0$ by $A\ge 1$, so that we may assume $x\in U_{\theta^k(0)}\setminus U_{\theta^{k-1}(0)}$ for some $k\ge 1$.
 Then $U_0\subseteq x+U_{\theta^{k+1}(0)}$ in light of
 \begin{align*}
 U_0=x-x+U_0\subseteq x + U_{\theta^k(0)} + U_{\theta^k(0)}\subseteq x+U_{\theta^{k+1}(0)}.
 \end{align*}
Since $\mu(x+U_{\theta^{k+1}(0)})\le A^{k+1}\mu(U_{0})$ by translation invariance of the Haar measure, we obtain
 \begin{align*}
 \sqrt[n]{P(x)} &=\frac1{A^{k}\mu(U_0)}\int_{U_0} 1 \dd\mu = \frac1{A^{k}\mu(U_0)}\int_{x+U_{\theta^{k+1}(k_0)}} \psi_0 \dd \mu \\
& \le \frac A{\mu(x+U_{\theta^{k+1}(k_0)})} \int_{x+U_{\theta^{k+1}(k_0)}} \psi_0 \dd \mu \le A\calm_\gr \psi_0(x).
 \end{align*}
With $\Om:=\sqrt[n]{\om}$ it thus follows
 \begin{align*}
 \|f\|_{\LR{p}_\om(\gr)}&\le \|f \cdot P^{-1}\|_{\LR{\infty}(\gr)} \|P\|_{\LR{p}_\om(\gr)} \le \rho_n(f) \|\sqrt[n]{P}\|_{\LR{np}_\Om(\gr)}^n \le A \rho_n(f) \|\calm_\gr \psi_0\|_{\LR{np}_\Om(\gr)}^n \\
 &\le A \rho_n(f)  \|\calm_\gr\|_{\LR{np}_\Om(\gr)\to\LR{np}_\Om(\gr)}^n\|\psi_0\|_{\LR{np}_\Om(\gr)}^n = A\rho_n(f) \|\calm_\gr\|_{\LR{np}_\Om(\gr)\to\LR{np}_\Om(\gr)}^n \|\om\|_{\LR{p}(U_0)}.
 \end{align*}
 This yields the assertion, since $\|\calm_\gr\|_{\LR{np}_\Om(\gr)\to\LR{np}_\Om(\gr)}^n\le [\Om]_{np}^{n(np)'}= [\om]_{p,(np)'/n}^{(np)'}\le [\om]_{p,r}^{(np)'}$ by Proposition \ref{mw_muckenhoupt_main} and the general properties of weights in Proposition \ref{mw_weight_prop}.
\end{proof}
We introduce the Fourier transform $\calf:\LR{1}(\gr)\to C_0(\widehat{\gr})$, where $\widehat{\gr}$ is the Pontryagin dual group of $\gr$.
Then $\calf:\LR{1}(\gr)\cap \LR{2}(\gr)\to C_0(\widehat{\gr})$ extends to an isometric isomorphism $\calf:\LR{2}(\gr)\to \LR{2}(\widetilde\gr)$.
We usually write $\hat f:=\calf f$ for $f\in \LR{1}(\gr)$.
We will also use the notation $\calf \LR{1}(\widehat\gr):=\set{\hat f: f\in \LR{1}(\gr)}$.
For $m\in \LR{\infty}(\widehat\gr)$ and $f\in \LR{2}(\gr)$ we introduce $T_m f:=\calf^{-1} m\hat f$.
For $p\in [1,\infty]$, $m\in \LR{\infty}(\widehat\gr)$, and a weight $\om$ on $\gr$ with $[\om]_p<\infty$ we write $m\in M_{p,\om}(\widehat{\gr})$ if $T_m$ originally defined on $\LR{p}_\om(\gr)\cap \LR{2}(\gr)$ extends to a bounded linear operator on $\LR{p}_\om(\gr)$.
In that case we write $\|m\|_{M_{p,\om}(\widehat{\gr}}:=\|T_m\|_{\LR{p}_\om(\gr)\to\LR{p}_\om(\gr)}$.
Consider
\begin{align*}
A(\gr)&:=\setc{f\in \LR{\infty}_{\mathrm{dec}}(\gr)}{\hat f \text{ has compact support}}.
\end{align*}
In order to have good approximation results, we assume the following on $\gr$.
Observe that by Lemma \ref{js41} we have $A(\gr)\subseteq \LR{p}_\om(\gr)$ for all $p\in (1,\infty)$ and all weights $\om$ with $[\om]_p<\infty$.
\begin{assumption}\label{ass_3}
There exist $c\in \R$ and a sequence $\seqkN{\varphi}\subseteq A(\gr)$ such that for all $f\in \LR{\infty}_{\mathrm{dec}}(\gr)$ it holds $\sup_{k\in\N}|\varphi_k\ast f|\le c\calm f$, $0\le \widehat{\varphi}_k\le 1$ and $\widehat{\varphi}_k \to 1$ locally uniformly.
\end{assumption}
\begin{lemma}\label{js40}
Let $\gr$ be a group subject to Assumptions \ref{mw_ass_ball} and \ref{ass_3}.
Then for all $p\in (1,\infty)$ and weights $\om$ with $[\om]_p<\infty$, it holds that $A(\gr)$ is a dense subset of $\LR{p}_\om(\gr)$.
\end{lemma}
\begin{proof}
Let $f\in \LR{p}_\om(\gr)$.
Since $\om^p\dd \mu$ is a regular measure on $\gr$ by dominated convergence, $\CRc{}(\gr)$ is dense in $\LR{p}_\om(\gr)$, see \cite[Appendix E.8]{Rud62}.
Hence, it suffices to assume that $f\in \CRc{}(\gr)\subseteq \LR{\infty}_{\mathrm{dec}}(\gr)$.
By assumption there are $c\in \R$ and $\seqkN{\varphi}\subseteq A(\gr)$ such that
\begin{align*}
\sup_{k\in\N}|\varphi_k\ast f|\le c\calm f \quad \text{and} \quad \widehat{\varphi}_k \to 1  \quad \text{locally uniformly.}
\end{align*}
Observe that $\varphi_k \ast f\in A(\gr)$ for all $k\in\N$, since $\LR{\infty}_{\mathrm{dec}}(\gr)$ is a convolution algebra and $\calf[\varphi_k\ast f]=\widehat\varphi_k \widehat f$ is compactly supported.
Moreover $\|\varphi_k\ast f - f\|_{\LR{2}(\gr)}=\|(\widehat \varphi_k-1)\widehat f\|_{\LR{2}(\gr)}\to 0$ by Plancherel's theorem and dominated convergence, so that in particular $\varphi_k\ast f \stackrel{\mu}{\longrightarrow} f$ in measure.
Since $[\om]_{p}<\infty$, we have $\calm f\in \LR{p}_\om(\gr)$ by Proposition \ref{mw_muckenhoupt_main}.
Thus the result follows by dominated convergence.
\end{proof}

From now on, we will assume that $\gr$ satisfies Assumptions \ref{mw_ass_ball} and \ref{ass_3}.
\begin{proposition}\label{js103}
Let $p\in(1,\infty)$, and assume that $\om$ is a weight on $\gr$ with $[\om]_p<\infty$.
Let $m\in \LR{\infty}(\widehat \gr)$.
Then $m\in M_{p,\om}(\widehat{\gr})$ if and only if there is $c\in\R$ such that for all $f,g\in A(\gr)$ we have
\begin{align}\label{js100a}
\left|\int_{\widehat{\gr}}m(\gamma)\widehat{f}(\gamma)\widehat{g}(\gamma)\mathrm{d}\gamma\right|\le c\|f\|_{L_\om^p(\gr)}\|g\|_{L_{\sigma}^{p'}(\gr)},
\end{align}
where $\sigma(x):=\om^{{-1}}(-x)$.
In this case $\|m\|_{M_{p,\om}(\widehat{\gr})}$ is the smallest constant such that \eqref{js100a} is valid.
\end{proposition}
\begin{proof}
Let $m\in M_{p,\om}(\widehat{\gr})$.
By Parseval's theorem and H\"older's inequality we have for all $f,g\in A(\gr)$
\begin{align*}
\left|\int_{\widehat{\gr}}m(\gamma)\widehat{f}(\gamma)\widehat{g}(\gamma)\mathrm{d}\gamma\right| &= \left|\int_{\gr} T_m f(x) g(-x)\mathrm{d}x\right| \\
&\le \|T_mf\|_{L_\om^p(\gr)}\|g\|_{L_{\sigma}^{p'}(\gr)} \\
&\le \|m\|_{M_{p,\om}(\widehat{\gr})}\|f\|_{L_\om^p(\gr)}\|g\|_{L_{\sigma}^{p'}(\gr)}.
\end{align*}
This shows estimate \eqref{js100a} with $c\le \|m\|_{M_{p,\om}(\widehat{\gr})}$.

\medskip

Conversely, let $c>0$ be such that estimate \eqref{js100a} is valid and let $f,g\in A(\gr)$.
Write $\tilde g(x):=g(-x)$.
By Parseval's formula it holds
\begin{align*}
\int_\gr T_mf(x) g(x) \dd x = \int_{\widehat{\gr}} m(\gamma) \widehat f(\gamma)\widehat {\tilde g}(\gamma)\dd\gamma,
\end{align*}
and thus
\begin{align*}
\left|\int_\gr T_mf(x)g(x)\dd x\right|\le c\|f\|_{\LR{p}_\om(\gr)}\|\tilde g\|_{\LR{p'}_\sigma(\gr)}=c\|f\|_{\LR{p}_\om(\gr)}\|g\|_{\LR{p'}_{\om^{-1}}(\gr)}
\end{align*}
for all $f,g\in A(\gr)$.
Since $[\om^{-1}]_{p'}=[\om]_p<\infty$, $A(\gr)$ is dense in $\LR{p}_\om(\gr)$ and $\LR{p'}_{\om^{-1}}(\gr)$ by Lemma \ref{js40}.
It follows that $T_m$ is bounded in $\LR{p}_\om(\gr)$ with $\|T_m\|_{\LR{p}_\om(\gr)\to\LR{p}_\om(\gr)}\le c$.
This shows the result.
\end{proof}

\begin{definition}
Let $\gr$ and $\grpH$ be locally compact abelian groups. Assume that $\Phi:\dualgrpH\to\dualgrp$ is a continuous homomorphism. By virtue of Pontryagin's duality theorem we may define the \emph{dual homomorphism of} $\Phi$ as the function $\widehat\Phi:\gr\to \grpH$ that satisfies
\begin{align}
(\widehat\Phi(x),\chi)=(x,\Phi(\chi)) \qquad \mbox{for all } x\in \gr, \chi\in \dualgrpH.
\end{align}
\end{definition}

For $f\in \LR{1}_{\mathrm{loc}}(\gr)$ and $z\in \gr$ we write $\tau_z f:=f(\cdot-z)$. 
\begin{proposition}\label{js100}
Let $\gr$ and $\grpH$ be locally compact abelian groups subject to Assumptions \ref{mw_ass_ball} and \ref{ass_3}.
Let $\Phi:\dualgrpH\to\dualgrp$ be a continuous homomorphism. Let $p\in(1,\infty)$ and suppose that $\om:\grpH\to(0,\infty)$ is a weight such that $[\om]_{p}<\infty$, $[\om\circ\widehat{\Phi}]_p<\infty$, and there is $c>0$ with $\|\cdot\|_{M_{p,\tau_u\om\circ\widehat\Phi}(\dualgrp)}\le c\|\cdot\|_{M_{p,\om\circ\widehat\Phi}(\dualgrp)}$ for all $u\in \grpH$.
If $m\in M_{p,\om\circ\widehat\Phi}(\dualgrp)\cap C(\dualgrp)$, then there holds $m\circ\Phi\in M_{p,\om}(\dualgrpH)$ and
\begin{align*}
\|m\circ\Phi\|_{M_{p,\om}(\dualgrpH)}\le c\|m\|_{M_{p,\om\circ\widehat\Phi}(\dualgrp)}.
\end{align*}
\end{proposition}
\begin{remark}\label{js101} 
\begin{enumerate}
\item\label{js101i} It always holds $c\ge 1$. Indeed, this follows from $\tau_u=\id$ for $u:=0\in \grpH$.
\item\label{js101ii} $p\in (1,\infty)$ and let $\om:\gr\to(0,\infty)$ be a weight function with $[\om]_p<\infty$. Then for all $z,z'\in \gr$ and $m\in L^\infty(\dualgrp)$ we have $\|m\|_{M_{p,\tau_{z'}\om}(\dualgrp)}=\|m\|_{M_{p,\tau_{z}\om}(\dualgrp)}$.
Indeed, since $T_m(\tau_z f)=\tau_z(T_m f)$, there holds
\begin{align*}
\|T_m f\|_{L_{\tau_{z}\om}^p(\gr)} &=\|\tau_{z'-z}(T_m f)\|_{L_{\tau_{z'}\om}^p(\gr)}= \|T_m(\tau_{z'-z} f)\|_{L_{\tau_{z'}\om}^p(\gr)} \\
&\le \|m\|_{M_{p,\tau_{z'}\om}(\dualgrp)}\|\tau_{z'-z} f\|_{L_{\tau_{z'}\om}^p(\gr)} = \|m\|_{M_{p,\tau_{z'}\om}(\dualgrp)}\|f\|_{L_{\tau_{z}\om}^p(\gr)},
\end{align*}
so that the claim follows by symmetry in $z$ and $z'$.
\item\label{js101iii} In particular, if $\widehat\Phi$ is surjective the assumption $\|\cdot\|_{M_{p,\tau_u\om\circ\widehat\Phi}}\le c\|\cdot\|_{M_{p,\om\circ\widehat\Phi}}$ is fulfilled with $c=1$. Indeed, in that case there is for all $u\in \grpH$ some $y\in \gr$ such that $\widehat\Phi(y)=u$ and hence $\tau_u\om\circ\widehat\Phi=\tau_y(\om\circ\widehat\Phi)$, so that choosing $z:=0$ and $z':=y$ in \ref{js101ii} yields $c=1$.
\end{enumerate}
\end{remark}

\medskip

\begin{lemma}\label{js001}
 Let $\Phi:\dualgrpH\to\dualgrp$ be a continuous homomorphism. For $m\in \calf \LR{1}(\dualgrp)\cap \LR{1}(\dualgrp)$ and $E\in \calf \LR{1}(\grpH)\cap \LR{1}(\grpH)$, it holds
 \begin{align*}
 \int_{\grp} E(\widehat\Phi(x))\widehat m(x)\dd x = \int_{\dualgrpH} m(\Phi(\chi))\widehat E(\chi)\dd\chi.
 \end{align*}
\end{lemma}
\begin{proof}
See Lemma B.1.3 in \cite{EdG77}.
\end{proof}

\begin{proposition}\label{js102}
 Theorem \ref{js100} is true if it is established for all $m\in M_{p,\om\circ\widehat\Phi}(\dualgrp)\cap \CRc{}(\dualgrp)$.
\end{proposition}
\begin{proof}
Let $m\in M_{p,\om\circ\widehat\Phi}(\dualgrp)\cap C(\dualgrp)$, and let $\seqkN{\varphi}\subseteq A(\gr)$ be as in Assumption \ref{ass_3}.
Then $\hat \varphi_k\in \CRc{}(\dualgrp)$, and we have for all $f,g\in A(\gr)$
\begin{align*}
\left|\int_{\dualgrp} m(\gamma)\widehat{\varphi}_k(\gamma) \widehat f(\gamma)\widehat g(-\gamma)\dd\gamma\right|&\le \|m\|_{M_{p,\om\circ\widehat\Phi}(\dualgrp)}\|\varphi_k\ast f\|_{\LR{p}_{\om\circ\widehat\Phi}(\gr)}\|g\|_{\LR{p'}_{(\om\circ\widehat\Phi)^{-1}}(\gr)} \\
&\le \|m\|_{M_{p,\om\circ\widehat\Phi}(\dualgrp)}\|\calm f\|_{\LR{p}_{\om\circ\widehat\Phi}(\gr)}\|g\|_{\LR{p'}_{(\om\circ\widehat\Phi)^{-1}}(\gr)} \\
&\le c(p,\gr)[\om\circ\widehat\Phi]_p^{p'}\|m\|_{M_{p,\om\circ\widehat\Phi}(\dualgrp)}\|f\|_{\LR{p}_{\om\circ\widehat\Phi}(\gr)}\|g\|_{\LR{p'}_{(\om\circ\widehat\Phi)^{-1}}(\gr)},
\end{align*}
so that in virtue of Proposition \ref{js103}
\begin{align*}
\|m\widehat\varphi_k\|_{M_{p,\om\circ\widehat\Phi}(\dualgrp)}\le c(p,\gr)[\om\circ\widehat\Phi]_p^{p'}\|m\|_{M_{p,\om\circ\widehat\Phi}(\dualgrp)}.
\end{align*}
The rest of the argument follows as in Lemma B1.1 of \cite{EdG77}.
\end{proof}

\begin{proof}[Proof of Theorem \ref{js100}]
By the same reasoning as in Lemma B.1.2 in \cite{EdG77} it suffices to assume $m\in \calf \LR{1}(\gr)\cap \LR{1}(\gr)$.
Let $h,k\in A(\grpH)$.
By Proposition \ref{js103} and Lemma \ref{js001} it suffices to establish
\begin{align*}
\left|\int_{\gr} \widehat m(x) (h\ast k)(\widehat\Phi(x))  \dd x\right|\le \|m\|_{M_{p,\om\circ\widehat{\Phi}}(\dualgrp)}\|h\|_{\LR{p}_\om(\grpH)}\|k\|_{\LR{p'}_\sigma(\grpH)}.
\end{align*}
For this purpose, let $\eps>0$.
Since $m\cdot (h\ast k)\circ\widehat\Phi\in \LR{1}(\gr)$, there is a compact $K\subseteq \gr$ such that
\begin{align*}
\int_{\gr\setminus K} |(\widehat m(x) (h\ast k)(\widehat\Phi(x))|\dd x \le \eps \|m\|_{M_{p,\om\circ\widehat{\Phi}}(\dualgrp)}\|h\|_{\LR{p}_\om(\grpH)}\|k\|_{\LR{p'}_\sigma(\grpH)}.
\end{align*}
Hence it suffices to construct a measurable function $v:\gr\to\C$ with $\1_K\le v\le 1$ such that
\begin{align}\label{js102e1}
\left|\int_{\gr} \widehat m(x) (h\ast k)(\widehat\Phi(x)) v(x)  \dd x\right|\le (1+\eps)\|m\|_{M_{p,\om\circ\widehat{\Phi}}(\dualgrp)}\|h\|_{\LR{p}_\om(\grpH)}\|k\|_{\LR{p'}_\sigma(\grpH)}.
\end{align}
By Theorem 2.6 in \cite{EiK17}, there are $f,g:G\to\C$ with $0\le f\ast g\le 1$, $f\ast g=1$ on $K$ and
\begin{align}\label{js102e2}
\|f\|_{\LR{p}(\gr)}\|g\|_{\LR{p'}(\gr)}\le 1+\eps.
\end{align}
Observe that
\begin{align*}
(h\ast k)(\widehat\Phi(x))&(f\ast g)(x)=\int_\grpH h(\widehat\Phi(x)-u)k(u) \mathrm{d}u \int_\gr f(x-y)g(y) \mathrm{d}y \\
&=\int_\gr\int_\grpH h(\widehat\Phi(x-y)-u)k(u+\widehat\Phi(y))f(x-y)g(y)\mathrm{d}u\mathrm{d}y \\
&=\int_\grpH \left[(\tau_uh\circ\widehat\Phi)f\right]\ast \left[(\tau_{-u}k\circ\widehat\Phi)g\right](x)\mathrm{d}u.
\end{align*}
Hence we may recast the left-hand side of \eqref{js102e1} with $v:=f\ast g$ as
\begin{align*}
\int_\grpH\left|\int_{\dualgrp} m(\gamma) \mathcal{F}\left[(\tau_uh\circ\widehat\Phi)f\right](\gamma)\cdot \mathcal{F}\left[(\tau_{-u}k\circ\widehat\Phi)g\right](\gamma)\mathrm{d}\gamma\right|\mathrm{d}u,
\end{align*}
which by Proposition \ref{js103} is estimated by
\begin{align*}
\int_\grpH \|m\|_{M_{p,\tau_u\om\circ\widehat\Phi}(\dualgrp)}\|(\tau_uh\circ\widehat\Phi) f\|_{L_{\tau_u\om\circ\widehat\Phi}^p(\gr)}\|(\tau_{-u}k\circ\widehat\Phi) g\|_{L_{\tau_{-u}\sigma\circ\widehat\Phi}^{p'}(\gr)}\mathrm{d}u,
\end{align*}
where we have used that
\begin{align*}
(\tau_{u}\om\circ\widehat\Phi)^{-1}(-x) &=\om^{-1}(\widehat\Phi(-x)-u)=\om^{-1}(-(\widehat\Phi(x)+u)) \\
&=\sigma(\widehat\Phi(x)+u)=(\tau_{-u}\sigma\circ\widehat\Phi)(x).
\end{align*}
Relying on the hypothesis $\|m\|_{M_{p,\tau_u\om\circ\widehat\Phi}(\dualgrp)}\le c\|m\|_{M_{p,\om\circ\widehat\Phi}(\dualgrp)}$ for all $u\in \grpH$, we arrive at
\begin{align*}
\Big|\int_{\gr} &\widehat m(x) (h\ast k)(\widehat\Phi(x)) (f\ast g)(x)  \dd x\Big| \\
&\le c\|m\|_{M_{p,\om\circ\widehat\Phi}(\dualgrp)} \int_\grpH \|(\tau_uh\circ\widehat\Phi) f\|_{L_{\tau_u\om\circ\widehat\Phi}^p(\gr)}\|(\tau_{-u}k\circ\widehat\Phi) g\|_{L_{\tau_{-u}\sigma\circ\widehat\Phi}^{p'}(\gr)}\mathrm{d}u \\
&\le c\|m\|_{M_{p,\om\circ\widehat\Phi(\dualgrp)}} \left(\int_\grpH \|(\tau_uh\circ\widehat\Phi) f\|_{L_{\tau_u\om\circ\widehat\Phi}^p(\gr)}^p \dd u\right)^{\frac1p} \left(\int_\grpH \|(\tau_{-u}k\circ\widehat\Phi) g\|_{L_{\tau_{-u}\sigma\circ\widehat\Phi}^{p'}(\gr)}^{p'}\mathrm{d}u\right)^{\frac1{p'}} \\
&= c\|m\|_{M_{p,\om\circ\widehat\Phi}(\dualgrp)}\left(\int_\grpH\int_\gr \left| h(\widehat\Phi(x)-u)f(x) \om(\widehat\Phi(x)-u)\right|^p\mathrm{d}x \mathrm{d}u\right)^{\frac1p} \\
&\qquad \times\left(\int_\grpH\int_\gr \left| k(\widehat\Phi(x)+u)g(x) \sigma(\widehat\Phi(x)+u)\right|^{p'}\mathrm{d}x\mathrm{d}u\right)^{\frac1{p'}} \\
&= c\|m\|_{M_{p,\om\circ\widehat\Phi}(\dualgrp)}\|h\|_{L_\om^p(\grpH)}\|f\|_{L^p(\gr)}\|k\|_{L_{\sigma}^{p'}(\grpH)}\|g\|_{L^{p'}(\gr)},
\end{align*}
where we have used Fubini and the translation and reflection invariance of the Haar measure in the last step.
By \eqref{js102e2}, this is the desired estimate \eqref{js102e1} with $v:=f\ast g$.
\end{proof}



\end{document}